\documentclass[preprint,12pt]{elsarticle}

\usepackage[margin=2cm]{geometry}
\usepackage[fleqn]{amsmath}
\usepackage{adjustbox}
\usepackage{amsfonts}  
\usepackage{dsfont}                  
\usepackage{amssymb}                    
\usepackage{graphicx}                   
\usepackage{epstopdf}
\usepackage[small,bf,hang]{caption}     
\usepackage{subcaption}
\usepackage{float}                      
\usepackage{flafter}                    
\usepackage{listings}                   
\usepackage{pdfpages}
\usepackage[nottoc]{tocbibind}			%
\usepackage{eurosym}                    
\usepackage{textcomp}                   
\usepackage{fancyhdr}                   
\usepackage{emptypage}
\usepackage{longtable}
\usepackage{array}
\usepackage{booktabs}
\usepackage{comment}
\usepackage{geometry}
\usepackage{pdflscape}
\usepackage{rotating}
\usepackage{siunitx}
\usepackage{xparse}
\usepackage{interval}
\usepackage{cancel}
\usepackage{tikz}
\usepackage{mathrsfs}
\usepackage{lscape}
\usepackage{tabularx}
\usepackage[T1]{fontenc}
\newcommand{\R}{\mathbb{R}}

\usepackage{longtable}
\usepackage[inline]{enumitem}
\setlist[enumerate,1]{label=\textit{\alph*)}}

\usepackage{natbib}
\usepackage{tabularx}
\usepackage{amsthm}
\usepackage{mathtools}
\usepackage{multirow}
\usepackage{hhline}
\usepackage{algorithm}
\usepackage{algorithmicx}
\usepackage{algpseudocode}
\usepackage{titlesec}
\usepackage{accents}
\usepackage[titletoc,toc,title]{appendix}
\usepackage{hyperref}
\usepackage{upgreek}
\newtheorem{thm}{Theorem}

	\newtheorem{prop}{Proposition}
	\theoremstyle{definition}
	
	\theoremstyle{remark}
	\newtheorem*{rem}{Remark}
	\newtheorem*{assm}{Assumptions}

\journal{
Computers \& Mathematics With Applications}

\begin{document}

\begin{frontmatter}



\title{A BDF B-spline Method for a Nonlocal Tumor Growth Model}


\author[inst1]{Bouhamidi Abderrahman}

\affiliation[inst1]{organization={Université du littoral côte d'opale},
            addressline={50 rue Ferdinand Buisson, CS 80699}, 
            city={Calais},
            postcode={62100}, 
            country={France}}

\author[inst2]{El Harraki Imad}
\author[inst1,inst2]{Melouani Yassine}

\affiliation[inst2]{organization={Ecole nationale supérieure des mines de Rabat},
            addressline={Av. Hadj Ahmed Cherkaoui - B.P. : 753}, 
            city={Rabat},
            postcode={10000}, 
            country={Morocco}}

\begin{abstract}
This paper presents a model for tumor growth using nonlocal velocity. We establish some results on the existence and uniqueness of the solution for a nonlocal tumor growth model. Many experiences show that tumor spheroid can be invariant by rotation and can guard the shape of the spheroid during the growth process in some particular cases. Here, we assume that the multiple components of the system are invariant by rotation. Then, we use the Backward Differentiation Formula (BDF) spline to solve the nonlocal system. To illustrate the effecienty of the proposed method, we performed numerical tests that simulate a tumor growth scenario. Such techniques may be used to provide informations on practical applications of the model.
\end{abstract}


\end{frontmatter}


\section*{Introduction}
		\footnotesize
			Mathematical modeling of tumor growth is a promising area of research that has a potential to improve our understanding of cancer and develop more effective treatments. By simulating the behavior of cancer cells and their interactions with their environment, researchers can gain insights into the factors that contribute to tumor growth and optimize treatment strategies. In fact, one of the primary motivations for studying tumor growth through mathematical modeling is to gain a deeper understanding of the underlying biological processes that drive cancer, such as genetic mutations, immune system response, and blood vessel formation. This understanding can then be used to identify potential targets for new cancer treatments.
			In recent years, an important increase in tumor growth models has been studied, see for instance \cite{belmiloudi_mathematical_2017,preziosi_cancer_2003,byrne2003modelling,lefebvre2017spatial,bresch_computational_2010,eladdadi_mathematical_2014} and references therein. However, a better understanding of this phenomenon is still needed due to the diversity of complex factors that play a role in cancer invasion and metastasis.\\ 
			Recent studies have shown that cancer cells can communicate through a variety of mechanisms, including direct cell-to-cell contact, secretion of signaling molecules, and extracellular vesicles \cite{das_molecular_2015,deisboeck2009collective}. This communication allows cells to share information about their environment and coordinate their behavior to promote tumor growth and survival. Furthermore, there is increasing evidence that collective cell migration is common during the invasion and metastasis of malignant tumors. This behavior has been observed and investigated in many biological areas, for example, predator spray systems, fishing, flocks of birds, and insect colonies \cite{lee_non-local_2001}. \\
			Most classical approaches to modeling tumor growth using balance equations or diffusive partial differential equations do not take these collective patterns into account \cite{belmiloudi_mathematical_2017,armbruster2006model}. However, experience over the past several years has shown that nonlocal balance laws present a very good way to model the behavior of collective and nonlocal movement of the population. In \cite{colombo2015hyperbolic}, a nonlocal predator-prey model has been studied and in \cite{colombo_class_2012,aggarwal_crowd_2016}, a class of nonlocal pedestrian models has been experienced. Nonlocal modeling has been used also in economy; we refer here to some nonlocal supply-chain models \cite{armbruster2006model,armbruster2006continuum}.\\
			In this paper, we are going to extend the model studied in \cite{lefebvre2017spatial}, to a nonlocal balance laws model. Some similar models have been investigated in \cite{kang_age-structured_2022,ramirez2021influence,fritz_local_2019,fritz_unsteady_2019,gerisch_mathematical_2008,bitsouni_non-local_2018} and references therein, where in \cite{kang_age-structured_2022,ramirez2021influence,fritz_local_2019,fritz_unsteady_2019} a nonlocal diffusion operator has been used, while in \cite{gerisch_mathematical_2008,bitsouni_non-local_2018} the nonlocality has been laid on the advection term using a nonlocal velocity.\\
			This paper is presented as following. In Section \ref{sec1}, we introduce a simplified nonlocal model for tumor growth, laying the groundwork for the subsequent analysis. Section \ref{sec2} is dedicated to establishing the existence and uniqueness of the solution, leveraging general results from \cite{keimer_existence_2017} and \cite{keimer_existence_2018}. In Section \ref{sec3}, we focus on a radial model derived from the initial multidimensional model, examining its specific characteristics and implications. Section \ref{sec4} devoted to the construction of a numerical scheme using the Cubic B-splines method for spatial discretization and the Backward Differentiation Formula (BDF) for temporal integration. Finally, in Section \ref{sec5}, we perform accuracy tests to illustrate the efficiency of the proposed method and perform a numerical tests that simulate tumor growth scenarios in the absence of medical treatment. This allows to provide informations on practical applications of the model. 
\section{A nonlocal continuum model}
	\label{sec1}
Mathematical modeling of solid tumor growth using continuum models has been extensively studied, see \cite{byrne_weakly_1999,byrne2003modelling,byrne_two-phase_2003,michel2018mathematical}, and references therein. This type of modeling is compatible with solid and secondary tumors, also for multicellular spheroids.\\
The general form of the equation describing the tumor behavior is as follows:
\begin{equation}
    \label{local1}
    \dfrac{\partial P}{\partial t}(t,x)+\operatorname{div}\left(\mathrm{v}(t,x) P(t,x)\right)=\text{Production}-\text{Death}, \\
				\end{equation}
    where $(t,x)\in[0,T]\times \R^{d}$, $T>0$, $d\geq 2$, $\mathrm{v}$ is the velocity vector, and $P$ is the density of proliferation cells. 
    In the practical tumor growth modeling the dimension $d$ is usually 1, 2 or 3.
    The velocity vector is commonly determined by using Darcy's law under the consideration of tumor cells as incompressible flow. The downside of this equation is that it does not take into account the collective behavior of tumor cells, cell-to-cell adhesion, and other types of regrouping \cite{cristini2009nonlinear}. Thus, we use a nonlocal type of modeling, taking into account the adhesion in one hand, between tumor cells, and in the other hand between healthy cells.
    We consider the following nonlocal equation of tumor growth:
    \begin{equation}
    \label{nonlocal1}
    \dfrac{\partial P}{\partial t}(t,x)+\operatorname{div}\Bigl[\mathrm{V}[\vec{\alpha_{p}},\gamma_{p},P](t,x) P(t,x)\Bigr]=H\left(M(t,x)\right) P(t,x)-a_{1} \lambda(t,x)M(t,x) P(t,x),
				\end{equation}
    where $\vec{\alpha_{p}}$ is the directional vector of the velocity, $\gamma_{p}$ is the kernel function of the proliferation, $M$ represents the concentration of oxygen and nutriment in the tissue, $\lambda$ is the density of the drug, while $a_{1}$ is the resistance of the drug by tumor cells and the function $H$ represents the growth factor. Let us remark that if $a_{1}=0$, the tumor is resistant to this type of medical treatment. Here we assume that the proliferation cells accept the treatment, and so we have $a_{1}>0$.\\
Here, the flux is assumed to be nonlocal, so the velocity vector $\mathrm{V}$ is expressed as follows:
\begin{equation}
	\label{nonlocALterm}
 \mathrm{V}[\vec{\alpha_{p}},\gamma_{p},P](t,x) =\vec{\alpha_{p}}\left( \gamma_{p}*P\right) (t,x)=\vec{\alpha_{p}}(t,x)\iint_{\R^d}\gamma_{p}(x-y)P(t,y)d y.
 \end{equation}
 
 As usual, see \cite{gerisch_mathematical_2008}, the kernel function $\gamma_p$ is chosen as a Gaussien type function. 
The production rate is controlled by the growth factor $H$, and by the medical treatment $\lambda$.\\
As in \cite{michel_mathematical_2018,lefebvre2017spatial}, we assume that the growth function $H$ is given as follows:
$$H(M) = \kappa\frac{1+\tanh \left(\delta\left(M-M_{threshold}\right)\right)}{2}$$
This function indicates that if the supplement concentration is greater than the hypoxia threshold $M_{th}$, the tumor grows by a factor of $\kappa$, while if this amount is below the hypoxia value, growth is suspended. The positive constant $\delta$ controls the growth of the hyperbolic tangent function.

We add here a second equation describing the movement of healthy cells under medical treatment, as follows:
\begin{equation}
    \label{nonlocalS}
    \dfrac{\partial S}{\partial t}(t,x)+\operatorname{div}\Bigl[\mathrm{V}[\vec{\alpha_{p}},\gamma_{p},P](t,x) P(t,x)\Bigr]=-a_{2} \lambda(t,x)M(t,x) S(t,x), 
		\end{equation}
    where $S$ represents the density of healthy cells, $a_{2}$ is a constant that denotes the secondary effects of medical treatment, $\vec{\alpha}_{s}$ is a radial directional vector, $\gamma_{s}$ is a Gaussian kernel. The nonlocal velocity for healthy cells is presented as follows:
    \begin{equation}
    	\label{nonlocALterm2}
    	\mathrm{V}[\vec{\alpha_{s}},\gamma_{s},S](t,x) =\vec{\alpha}_{s}(t,x)\left( \gamma_{s}*S\right).
    \end{equation}
    The concentration of nutrients and oxygen is a crucial detail in tumor growth. The increase in this concentration is controlled by several factors. One of the main factors is angiogenesis, which is a fundamental biological process through which new blood vessels are formed from preexisting ones. In the context of tumor development, angiogenesis is often stimulated by signaling molecules and growth factors produced by cancer cells themselves. This process leads to the creation of an extensive network of blood vessels that infiltrate the tumor mass. This process can be omitted here under the condition that tumor cells have not yet reached the vascular stage. However, healthy cells are capable of producing and stimulating new resources of oxygen and nutrients when needed. However, excessive consumption of oxygen and nutrients by cancer cells causes a lack of supplements in the area.

    Finally, we add a third equation that describes the concentration of oxygen and nutrients as following:
  \begin{equation}
      \label{vasc}
      \dfrac{\partial M}{\partial t}(t,x)-D\Delta M(t,x)=M_{s}S(t,x)(1-M(t,x))-\eta M(t,x) P(t,x),
  \end{equation}
where $M_{s}$ is the rate at which healthy cells can create new blood vessels, $\eta$ is the rate of consumption of oxygen and nutrients by cancer calls, while $D$ denotes the diffusion of the vascularisation in the tissue.

    We assume that the domain of study is large enough such that there is no interaction between tumor cells and the boundary. In light of that and without loss of generality, we assume that the domain of our study is the open disk of center zero and radius $R>0$, denoted as  $$\Omega_{R}=\{x\in \R^d | \quad\|x\|_{2}< R\},$$ 
  where $\|\cdot\|_{2}$ is the usual euclidean norm in $\R^d$, and let $\partial \Omega_{R}$ be the boundary of $\Omega_{R}$. \\
  As there is no interaction of tumor cells with the boundary, we have 
  \begin{equation*}
      P=0 \quad \quad\text{in} \quad \partial \Omega_R.
  \end{equation*}
  In the other hand we assume that the density of healthy cells $S$ and the concentration of oxygen and nutrients $M$ remain constant throughout the boundary; consequently we have
  \begin{equation*}
      S=S_{R},\quad M=M_{R}  \quad\text{in} \quad \partial \Omega_{R}.
  \end{equation*}
  With $S_{R}$ and $M_{R}$ are two fixed positive values.\\
  We restrict the convolution terms $\gamma_{p}*P$ and $\gamma_{s}*S$ inside the computational domain, therefor we replace \ref{nonlocALterm} and \ref{nonlocALterm2} by
 \begin{align}
 	\mathrm{V}[\vec{\alpha_{p}},\gamma_{p},P](t,x)&=\vec{\alpha_{p}}(t,x)\int_{\R^d}\mathds{1}_{ \Omega_{R}}\gamma_{p}(x-y)P(t,y)d y=\vec{\alpha_{p}}(t,x)\int_{ \Omega_{R}}\gamma_{p}(x-y)P(t,y)d y \quad \text{and}\\
 \mathrm{V}[\vec{\alpha_{s}},\gamma_{s},S](t,x)&=\vec{\alpha_{s}}(t,x)\int_{\R^d}\mathds{1}_{ \Omega_{R}}\gamma_{s}(x-y)S(t,y)d y=\vec{\alpha_{s}}(t,x)\int_{ \Omega_{R}}\gamma_{s}(x-y)S(t,y)d y
 \end{align}
  To ensure that the growth process remains in $ \Omega_{R}$ and there is no interference with the boundary, let $R^{\prime}>0$ with $R^{'}<R$, there exists a function $m : \R^d \rightarrow \R$ that is infinitely differentiable, satisfying:
  \begin{equation*}
      m(x)=\begin{cases}
          1 \quad \text{if} \quad \|x\|_{2} \leq R^{\prime},\\
          0 \quad \text{if} \quad \|x\|_{2} \geq R.\\
      \end{cases}
  \end{equation*}

  We multiply the right-hand side of the equations of proliferative cells and healthy ones by the function $m$. We regroup equations (\ref{nonlocal1}), (\ref{nonlocalS}), and (\ref{vasc}) with their boundary conditions as follows:\\
  {\footnotesize
  \begin{equation}
  \label{model}
			    	\begin{cases}
			    		\vspace{0.2cm}
			    		\dfrac{\partial P}{\partial t}(t,x)+\operatorname{div}\Bigl[\mathrm{V}[\vec{\alpha_{p}},\gamma_{p},P](t,x) P(t,x)\Bigr]=m(x)\Bigl(H(M(t,x)) P(t,x)-a_{1} \lambda(t)M(t,x) P(t,x)\Bigr), \quad &(t,x)\in[0,T]\times  \Omega_{R}\\
			    		\vspace{0.2cm}
			\dfrac{\partial S}{\partial t}(t,x)+\operatorname{div}\Bigl[\mathrm{V}[\vec{\alpha_{p}},\gamma_{p},P](t,x) P(t,x)\Bigr]=-a_{2} m(x) \lambda(t)M(t,x) S(t,x),&(t,x)\in[0,T]\times  \Omega_{R}  \\
			    		\vspace{0.2cm}
			    		\dfrac{\partial M}{\partial t}(t,x)-D\Delta M(t,x) =M_{s}S(t,x)(1-M(t,x))-\eta M(t,x) P(t,x), &(t,x)\in[0,T]\times  \Omega_{R}\\
         \vspace{0.2cm}
         P(t,x)=0, \quad S(t,x)=S_R, \quad M(t,x)=M_R, &(t,x)\in[0,T]\times \partial \Omega_{R}\\
         \vspace{0.2cm}
                      P(0,x)=P_{0}(x), \quad
                      S(0,x)=S_{0}(x),\quad
                      M(0,.)=M_{0}(x), &x\in   \Omega_{R}.
        	\end{cases}
        	 \end{equation}

 Let $U=(P,S,M)$, we rewrite the system of equations \eqref{model} into the following form:  
       
       \begin{equation}
       	\label{equationU}
       	\left\{
       	\begin{array}{ccc}
       		\dfrac{\partial U}{\partial t}(t,x)+\mathcal{A}\left(U\right)(t,x) =G(U)(t,x), \quad &(t,x)\in[0,T]\times  \Omega_{R}\\  
       		U(t,x) =U_{R}, \quad &(t,x)\in[0,T]\times \partial \Omega_{R}\\ 
       		U(0,x)=U_{0}(x). \quad & x \in  \Omega_{R}
       	\end{array}
       	\right.
       \end{equation}
with 
\begin{equation*}
\resizebox{\textwidth}{!}{$
\displaystyle
    \mathcal{A}\left(U\right)(t,x)=\left[\operatorname{div}\left(\vec{\alpha_{p}}(t,x)\left(\int_{ \Omega_{R}}\gamma_{p}(x-y)P(t,y)d y \right) P(t,x)\right),\operatorname{div}\left(\vec{\alpha_{s}}(t,x)\left(\int_{ \Omega_{R}}\gamma_{s}(x-y)S(t,y)d y \right)S(t,x)\right),-D\Delta M\right]$},
\end{equation*}
 \begin{equation*}
       	G(U)=\begin{pmatrix}g_{1}(U)
       		\\ g_{2}(U)
       		\\ g_{3}(U) \end{pmatrix}=\begin{pmatrix} m P(H(M)-a_{1} \lambda M )
       		\\ -m a_{2} \lambda  M S
       		\\m_{s}S(1-M)-\eta M P \end{pmatrix},\quad U_{R}=(0,S_{R},M_{R})\quad \text{and}\quad U_{0}=(P_0,S_{0},M_{0}).
       \end{equation*}

       \section{Existence and uniqueness of solution }	\label{sec2}

       Our aim in this section is to prove the existence and uniqueness of solutions for the system outlined in equation (\ref{model}).
       We begin by examining the local dynamics through the isolation of the convolution term and leveraging established principles from the theory of semilinear evolution equations. Following this, we integrate insights from the theory of nonlocal balance equations, as elaborated in \cite{keimer_existence_2017} and \cite{keimer_existence_2018}, to affirm the existence and uniqueness of the solutions for the system mentioned in equation (\ref{nonlocal1}).\\
       Before proceeding with the existence and uniqueness proofs, we need to define the following spaces.
     
     Let $L^{\infty}(\Omega_R)$ be the space of essentially bounded measurable functions on $\Omega_R$, equipped with the norm:
    \[
    \|f\|_{L^{\infty}(\Omega_R)} = \operatorname{ess\,sup}_{x \in \Omega_R} |f(x)|.
    \]
     Let $C(\Omega_R)$ be he space of continuous functions on $\Omega_R$, with the uniform norm:
    \[
    \|f\|_{C(\Omega_R)} = \sup_{x \in \Omega_R} |f(x)| .
    \]
    Let $C^1_b(\Omega_R)$ be the space of continuously differentiable functions on $\Omega_R$ with bounded derivatives, normed by:
    \[
    \|f\|_{C^1_b(\Omega_R)} = \|f\|_{C(\Omega_R)} + \|\nabla f\|_{C(\Omega_R)}.
    \]
   For a Banach space $(X,\|\,.\,\|_X)$, let $C([0,T]; X)$ denote the space of continuous functions from $[0,T]$ to $X$, with the norm:
    \[
    \|f\|_{C([0,T]; X)} = \sup_{t \in [0,T]} \|f(t)\|_X  .
    \]
    Let $L^1([0,T]; X)$ be the space of Bochner integrable functions from $[0,T]$ to $X$, with the norm:
    \[
    \|f\|_{L^1([0,T]; X)} = \int_0^T \|f(t)\|_X dt.
    \]
    Let $W^{2,1}(\Omega_R)$ be the Sobolev space defined as:
 \[
W^{2,1}(\Omega_R) = \{u \in L^1(\Omega_R) : D^\alpha u \in L^1(\Omega_R) \quad\text{for all}\quad |\alpha| \leq 2\}
\]
       \subsection{Existence and uniqueness of the solution of the local system}

       Let $w_{p}$ and $w_{s}$ be two fixed functions belonging to $C\left([0,T],C^{1}_{b}(\R^d)\right)$, and let $w:=(w_{p},w_{s})$.\\
       We consider the associated local system to \eqref{equationU} and written as follows:
       \begin{equation}
       	\label{localequation}
       	\left\{
       	\begin{array}{ccc}
       		\dfrac{\partial U_{w}}{\partial t}(t,x)+\mathcal{A}_{w}\left(U_{w}\right)(t,x) =G(U_{w})(t,x), \quad &(t,x)\in[0,T]\times  \Omega_{R}\\  
       		U_{w}(t,x) =U_{R}, \quad &(t,x)\in[0,T]\times \partial \Omega_{R}\\ 
       		U(0,x)=U_{0}(x). \quad & x \in  \Omega_{R}
       	\end{array}
       	\right.
       \end{equation}
       with $\mathcal{A}_{w}\left(U\right)=\left[\operatorname{div}\left(\vec{\alpha_{p}}(t,x)w_{p}(t,x) P_{w}(t,x)\right),\operatorname{div}\left(\vec{\alpha_{s}}(t,x)w_{s}(t,x) S_{w}(t,x)\right),-D\Delta M\right]$, and
       \begin{equation*}
       	G(U_{w})=\begin{pmatrix}g_{1}(U_{w})
       		\\ g_{2}(U_{w})
       		\\ g_{3}(U_{w}) \end{pmatrix}=\begin{pmatrix} m P_{w}(H(M_{w})-a_{1} \lambda M_{w} )
       		\\ -m a_{2} \lambda  M_{w} S_{w}
       		\\M_{s}S_{w}(1-M_{w})-\eta M_{w} P_{w} \end{pmatrix}
       \end{equation*}
       and we have $U_{R}=(0,S_{R},M_{R})$ and $U_{0}=(P_0,S_{0},M_{0})$.

To ensure that the solution remains positive and bounded within the domain, we make the following assumptions regarding initial conditions and velocities. 
       
       \begin{assm}
       	\label{assm1}
       	We assume that:
       	\begin{itemize}
       	    
       		\item[\textbf{(A1)}] $U_{0} \in C^1( \Omega_{R})^3:=C^1( \Omega_{R})\times C^1( \Omega_{R}) \times C^1( \Omega_{R})$,
       		\item[\textbf{(A2)}]  $U_{0}=(P_0,S_{0},M_{0}) \geq 0$, which means $P_0\geq 0$, $S_0 \geq 0$, and $M_0\geq 0$,
       		\item[\textbf{(A3)}]  $U_{0}-U_{R}$ is compactly supported in $ \Omega_{R}$,
       		\item[\textbf{(A4)}]   $\vec{\alpha_{k}} \in C^2\left([0,T]\times \Omega_{R};\R^d\right)$\quad $k=p,s.$
       		\item[\textbf{(A5)}]  $\vec{\alpha_{s}}$ is compactly supported in $[0,T]\times \Omega_{R}$,
       		\item[\textbf{(A6)}]  $\lambda\in C([0,T])$.
       	\end{itemize}
       \end{assm}
       The Banach space $L^\infty(\Omega_R)^3:=L^\infty(\Omega_R)\times L^\infty(\Omega_R) \times L^\infty(\Omega_R)$  is equipped here with the norm:
       \begin{equation*}
            \|f\|_{L^\infty(\Omega_R)^3}=\|f_{1}(t)\|_{L^{\infty}(\Omega_{R})}+\|f_{2}(t)\|_{L^{\infty}(\Omega_{R})}+\|f_{3}(t)\|_{L^{\infty}(\Omega_{R})}\qquad\text{for all}\quad f=(f_{1},f_{2},f_{3})\in L^\infty(\Omega_R)^3.
       \end{equation*}
       \begin{thm}
       	\label{therom1}
       	Let the assumptions (A1)-(A6) hold. Then there exists a maximal time $T_{max}$ such that equation (\ref{localequation}) has a unique positive solution in $C\left([0,T],L^\infty(\Omega_{R})^3\right) $. Furthermore, $U(t,\cdot)_{w}-U_{R}$ is a compactly supported function in $\Omega_{R}$ for all $t\leq T_{max}$.
       \end{thm}
       \begin{proof}
       	To establish the existence and uniqueness of (\ref{localequation}), we apply the characteristic method. First, we ensure that there is no interaction between the tumor and the boundary by selecting $T$ so that the characteristics are appropriately defined within $\Omega_{R}$. This is guaranteed by the fact that $U_{0}-U_{R}$ and the velocity vector of healthy cells have a compact support in $\Omega_{R}$. Let $0<R_{0}<R$ be such that \qquad 
       		$\operatorname{supp}\left( U_{0}\right) \cup \operatorname{supp}\left( \vec{\alpha_{p}}\right)\subset \Omega_{R_{0}}$, 
       	
       	and we let \begin{equation}
       		R_{i}=R_{0}+\frac{i(R-R_{0})}{3} \qquad \text{for} \quad i=1,2,3.
       	\end{equation}
       	In order to prevent the characteristic curves from exiting $\Omega_{R}$, we proceed under the assumption that $T\leq T_{\max}$ is sufficiently small such that
       	\begin{equation}
       		\label{timeT}
       		T<T^*_1:=\dfrac{R-R_{0}}{3\max\{\| \vec{\alpha_{p}}w_p\|_{C\left([0,T_{\max}],L^{\infty}(\Omega_{R})\right)},\|\vec{\alpha_{s}}w_s\|_{C\left([0,T_{\max}],L^{\infty}(\Omega_{R})\right)}\}}
       	\end{equation}
        with $T_{\max}$ is a large fixed value.
       	The characteristics are well defined for all $t,s\in [0,T]$ and $x\in \Omega_{R_{2}}$ as follows 
       	\begin{equation}
       		\left\{\begin{array}{ll}
       			\dfrac{\partial X[t,x]_{w_{p}}(s)}{\partial s} &=\vec{\alpha_{p}}(s, X[t,x]_{w_{p}}(s)) w_p(s, X[t,x]_{w_{p}}(s)), \\
       			X[t,x]_{w_{p}}(t) & =x.
       		\end{array}\right.
       	\end{equation}
       	And
       	\begin{equation}
       		\left\{\begin{array}{ll}
       			\dfrac{\partial X[t,x]_{w_{s}}(s)}{\partial s} &=\vec{\alpha_{s}}(s, X[t,x]_{w_{s}}(s)) w_p(s, X[t,x]_{w_{s}}(s)), \\
       			X[t,x]_{w_{s}}(t) & =x.
       		\end{array}\right.
       	\end{equation}
       In fact, let $x\in \Omega_{R_{2}}$. By using the condition on the final time $T$ we have for $k=p,s$ :
       \begin{align*}
       	\|X[t,x]_{w_{k}}(s)\|_{2}&\leq \|x\|_2+\|X[t,x]_{w_{k}}(s)-x\|_{2}\\
       	&\leq R_{2}+\int^{s}_{t}\|\vec{\alpha_{k}}(s,X[t,x]_{w_{k}}(s))w_p(s, X[t,x]_{w_{k}}(s))\|_{2}ds\\
       	&\leq R_{0}+\frac{2}{3}(R-R_{0})+T\|\vec{\alpha_{k}}w_k\|_{C\left([0,T],L^{\infty}(\Omega_{R})\right)}\\
       	&<R_{0}+\frac{2}{3}(R-R_{0})+\frac{1}{3}(R-R_{0})=R.
       \end{align*}
       Using the semi-explicit form for the local balance law in \cite{keimer_existence_2018} we have for $(t,x)\in [0,T]\times \Omega_{R_{2}}$
       \begin{align*}
       	P_{w}(t,x)&=P_0\left(X[t,x]_{w_{p}}(0)\right) \operatorname{det}\left(\mathrm{D}X[t,x]_{w_{p}}(0)\right)+\int_{0}^{T}\operatorname{det}\left(\mathrm{D}X[t,x]_{w_{p}}(s)\right)g_{1}(U_{w})\left(s,X[t,x]_{w_{p}}(s)\right)ds\\
       	S_{w}(t,x)&=S_0\left(X[t,x]_{w_{s}}(0)\right) \operatorname{det}\left(\mathrm{D}X[t,x]_{w_{s}}(0)\right)+\int_{0}^{T}\operatorname{det}\left(\mathrm{D}X[t,x]_{w_{s}}(s^\prime)\right)g_{2}(U_{w})\left(s^\prime,X[t,x]_{w_{s}}(s^\prime)\right)ds^\prime.
       \end{align*}
       with 
       \begin{equation}
       \label{det}
       	\operatorname{det}\left(\mathrm{D}X[t,x]_{w_{k}}(s)\right)=
       	\exp\left(\int_{t}^{s}\operatorname{div\left(\vec{\alpha_{k}}(s,X[t,x]_{w_{k}}(s))w_p(s, X[t,x]_{w_{k}}(s))\right)ds}\right) \qquad \text{for} \quad k=p,s.
       \end{equation}
       Now, letting $t,s\in [0,T]$, and letting $x\in \Omega_{R_{2}}\backslash \Omega_{R_{1}}$, we have then
       \begin{align*}
       	\left\|X[t,x]_{w_{k}}(s)\right\|_2 &\geq \left|\left\|x\right\|_2-\left\|X[t,x]_{w_{k}}(s)-x\right\|_2\right| ,\qquad k=p,s\\
       	&>\left|R_{1}-\frac{1}{3}(R-R_{0})\right|=R_{0}.
       \end{align*}
       
       Using the assumption that initial functions are compactly supported in $\Omega_{R_{0}}$, we deduce that $P_{0}\left(X[t,x]_{w_{p}}(0)\right)=0$ and $S_{0}\left(X[t,x]_{w_{s}}(0)\right)=S_{R}$.\\
       Furthermore, the function $m$ ensures that the right-hand side is compactly supported in $\Omega_{R}$, by choosing $R^{\prime}\leq R_{0}$ and using the assumption of the compactness of the support of $\vec{\alpha_{s}}$, we get $P_{w}(t,x)=0$ and $S_{w}(t,x)=S_{R}$ for all $(t,x)\in [0,T]\times \Omega_{R_{2}}\backslash \Omega_{R_{1}}$. Using equation (\ref{localequation}), we can smoothly extend $P$ and $S$ to be constant in $\Omega_{R}\backslash \Omega_{R_{2}}$.\\
       For the equation of nutrients $M$, it is a semi-linear parabolic equation, therefore, by using the maximum principle we deduce that $U_{w}-U_{R}$ is compactly supported in $\Omega_{R}$. 
       Now we prove the local existence and uniqueness for Equation (\ref{localequation}). To do this, we employ a classical approach for semilinear evolutionary equations using Banach's fixed-point theorem. This approach is demonstrated in works such as \cite{zheng2004nonlinear,cazenave1998introduction}.
       
       In the first step, we show that the right hand side is locally lipschitzien. In fact, let $L>0$, for $U_{w}=\left( P_{w},S_{w},M_{w}\right)$ and $V_{w}=\left( P^{\prime}_{w},S^{\prime}_{w},M^{\prime}_{w}\right)$ in $C\left([0,T],L^\infty(\Omega_{R})^3\right) $ such that $\|U_{w}\|_{C\left([0,T],L^\infty(\Omega_{R})^3\right)}$,$\|V_{w}\|_{C\left([0,T],L^\infty(\Omega_{R})^3\right)}\leq L$, we have
       \begin{equation*}
            \|G(U_{w}(t))-G(V_{w}(t))(t)\|_{L^\infty(\Omega_{R})^3}=\|g_{1}(U_{w}(t))-g_{1}(V_{w}(t))\|_{\infty}+\|g_{2}(U_{w}(t))-g_{2}(V_{w}(t))\|_{\infty}+\|g_{3}(U_{w}(t))-g_{3}(V_{w}(t))\|_{\infty}.
       \end{equation*}
By using the fact that $\|H(M)\|_{C\left([0,T],L^\infty(\Omega_{R})\right)}\leq \kappa$ together with following inequalities:
$$\|g_{1}(U_{w}(t))-g_{1}(V_{w}(t))\|_{\infty}\leq \kappa\|P_{w}(t)-P^{\prime}_{w}(t)\|_{\infty}+ a_{1}L \left( \|P_{w}(t)-P^{\prime}_{w}(t)\|_{\infty}+\|M_{w}(t)-M^{\prime}_{w}(t)\|_{\infty}\right), $$
$$\|g_{2}(U_{w}(t))-g_{2}(V_{w}(t))\|_{\infty}\leq a_{2}L\left( \|S_{w}(t)-S^{\prime}_{w}(t)\|_{\infty}+\|M_{w}(t)-M^{\prime}_{w}(t)\|_{\infty}\right) +M_{s}M_{R}\|S_{w}(t)-S^{\prime}_{w}(t)\|_{\infty},$$
and 
   $$\|g_{2}(U_{w}(t))-g_{2}(V_{w}(t))\|_{\infty}\leq M_{s}L\left( \|S_{w}(t)-S^{\prime}_{w}(t)\|_{\infty}+\|M_{w}(t)-M^{\prime}_{w}(t)\|_{\infty}\right) +\eta L\left( \|P_{w}(t)-P^{\prime}_{w}(t)\|_{\infty}+\|M_{w}(t)-M^{\prime}_{w}(t)\|_{\infty}\right),$$
   we get
   \begin{equation}
       \|G(U_{w}(t))-G(V_{w}(t))\|_{L^\infty(\Omega_{R})^3}\leq C(L)\|U_{w}(t)-V_{w}(t)\|_{\infty},
   \end{equation}
     where $C(L)$ is a positive function depending on the constant $L$. Finally, by taking the maximum over time $t\in[0,T]$, we get the result.

       For the next step, we consider the following constants: 
       \begin{equation}
       \label{cv}
       	C_{k}=\|\vec{\alpha_{k}}\|_{\infty}\|\nabla w_{k}\|_{\infty}+\|\operatorname{div}\left( \vec{\alpha_{k}}\right) \|_{\infty}\|w_{k}\|_{\infty} \qquad\text{and}\quad C_{v}=e^{T\operatorname{max}\left\lbrace C_{p},C_{s}\right\rbrace}\geq 1.
       \end{equation}
       Let $K=2LC_{v}$. We denote by
       \begin{equation}
       	E=\{f\in C\left( [0,T],L^{\infty}(\Omega_{R})^3\right) \space | \enspace \|f\|_{C\left( [0,T],L^{\infty}(\Omega_{R})^3\right)}\leq K\},
       \end{equation}
       the closed ball of radius $K$ (closed convex) subset of the Banach space $C\left( [0,T],L^{\infty}(\Omega_{R})^3\right)$.
    
       Let $U_{w}\in E$, we consider the following functional  
       \begin{equation*}
       	\Phi(U_{w})(t)=\begin{pmatrix} P_0\left(X[t,x]_{w_{p}}(0)\right) \operatorname{det}\left(\mathrm{D}X[t,x]_{w_{p}}(0)\right)+\int_{0}^{T}\operatorname{det}\left(\mathrm{D}X[t,x]_{w_{p}}(s)\right)g_{1}(U_{w})\left(s,X[t,x]_{w_{p}}(s)\right)ds\\
       		S_0\left(X[t,x]_{w_{s}}(0)\right) \operatorname{det}\left(\mathrm{D}X[t,x]_{w_{s}}(0)\right)+\int_{0}^{T}\operatorname{det}\left(\mathrm{D}X[t,x]_{w_{s}}(s^\prime)\right)g_{2}(U_{w})\left(s^\prime,X[t,x]_{w_{s}}(s^\prime)\right)ds^\prime\\
       		\mathcal{T}(t)M_{0}+\int_{0}^{T}\mathcal{T}(t-s)g_{2}(U_{w})(s)d s 
       	\end{pmatrix},
       \end{equation*}
       where  $t\rightarrow \mathcal{T}(t)$ is the semigroup of the heat operator $A = -D\Delta$, with the domain:
$$D(A) = \{u \in W^{2,1}(\Omega_R) : Au \in L^{\infty}(\Omega_R), u|_{\partial \Omega_R} = M_R\}.$$
The semigroup $t\rightarrow \mathcal{T}(t)$ satisfies that:
       \begin{equation}
       \label{semig}
       	\|\mathcal{T}(t)M_0\|_{C([0,T],L^\infty(\Omega_R))}\leq \|M_0\|_{L^\infty(\Omega_R)}\qquad\text{for all}\,\, t\in [0,T],
       \end{equation}
       see \cite{brezis2010functional} for more details. 
       First, let us show that $\Phi(E) \subseteq E$. Let $U_{w}\in E$, we have
       \footnotesize
       \begin{align*}
       	\displaystyle
        \left\|\Phi\left(U_{w}\right)(t)\right\|_{\infty}=& \left\| P_0\left(X[t,x]_{w_{p}}(0)\right) \operatorname{det}\left(\mathrm{D}X[t,x]_{w_{p}}(0)\right)+\int_{0}^{T}\operatorname{det}\left(\mathrm{D}X[t,x]_{w_{p}}(s)\right)g_{1}(U_{w})\left(s,X[t,x]_{w_{p}}(s)\right)ds\right\|_{\infty} \\
       	&+\left\| S_0\left(X[t,x]_{w_{s}}(0)\right) \operatorname{det}\left(\mathrm{D}X[t,x]_{w_{s}}(0)\right)+\int_{0}^{T}\operatorname{det}\left(\mathrm{D}X[t,x]_{w_{s}}(s^\prime)\right)g_{2}(U_{w})\left(s^\prime,X[t,x]_{w_{s}}(s^\prime)\right)ds^\prime\right\|_{\infty} \\
       	&+\left\| \mathcal{T}(t)M_{0}+\int_{0}^{T}\mathcal{T}(t-s)g_{3}(U_{w})(s)d s\right\|_{\infty}.
        \end{align*}
        
        By using \eqref{det},\eqref{cv},\eqref{semig} and estimates from \cite{keimer_existence_2018}, we obtain
        \begin{equation}
\left\|\operatorname{det}\left(\mathrm{D}X[t,x]_{w_{k}}(s)\right)\right\|_{C([0,T],L^\infty(\Omega_R))} \leq C_v
       	 \qquad \text{for} \quad k=p,s.
       \end{equation}
       Then, it follows that
        \begin{equation}
\|\Phi(U_{w})(t)\|_{\infty}\leq C_v\left( \|P_{0}\|_{\infty}+\|S_{0}\|_{\infty}\right) +\|M_{0}\|_{\infty}+T\left[C_{v}\left( \|g_{1}(U_{w})(t)\|_{\infty}+\|g_{2}(U_{w})(t)\|_{\infty}\right)+\|g_{3}(U_{w})(t)\|_{\infty}\right].
       \end{equation}
       Using the lipschitz property of $G$ and $G(0)=0$, we have
       \begin{align*}
       	\|\Phi(U_{w})(t)\|_{\infty}&\leq C_v\left[\left( \|P_{0}\|_{\infty}+\|S_{0}\|_{\infty}+\|M_{0}\|_{\infty}\right) +TC(L)\left( \|g_{1}(U_{w})(t)\|_{\infty}+\|g_{2}(U_{w})(t)\|_{\infty}+\|g_{3}(U_{w}(t))\|_{\infty}\right)\right] \\
       	&\leq C_{v}L(1+TC(L)).
       \end{align*}
       Then, by choosing 
       \begin{equation}\label{defT2}
          T \leq T^*_2:=\min\left\{T^*_1, \dfrac{1}{2 C(L)}\right\},
       \end{equation}  and taking the maximum over time, we deduce that $\Phi(E) \subseteq E$.
       
       Now, let $U_{w},U_{w}^{\prime} \in E$, we have
       \begin{align*}
       	\left \|\Phi(U_{w})-\Phi(U_{w}^{\prime})\right\|_{E}\leq T C_{v} \left \|U_{w}-U_{w}^{\prime}\right\|_{E} \leq\frac{1}{2}\left\|(U_{w}-U_{w}^{\prime}\right\|_{E}.
       \end{align*} 
       Therefore, $\Phi$ is a contraction in $E$ with Lipschitz constant 1/2, and so $\Phi$ has a fixed point $U_{w}\in E$, which ensures the existence and uniqueness of a solution of (\ref{localequation}).
       \end{proof}
       \subsection{Existence and uniqueness of the nonlocal system}
       Now we turn back into the nonlocal system (\ref{nonlocal1}), we use Banach's fixed theorem once again as in \cite{keimer_existence_2018}, therefore, we announce the main theorem
\begin{thm} Let the assumptions (A1)-(A6) hold, and assume that $\gamma_{p},\gamma_{s}\in C^{1}_{b}(\Omega_{R})$. Then the system of equations (\ref{nonlocalS}) has a unique solution $U=(P,S,M)$ in $C\left([0,T],L^\infty(\Omega_{R})^3\right) $.
       \end{thm}
       \begin{proof}
       	In the proof, we follow the lines of (\cite{keimer_existence_2018}) adapted to the case of a system of coupled nonlocal equations. \\
       	Let define the following constants:
       	\begin{align*}
       		N_{p}&:=\displaystyle \left\| \gamma_{p}\right\| _{C\left([0,T]\times \Omega_{R}\right)}\left(\left\| P_{0}\right\| _{L^{1} (\Omega_{R})}+\sup_{w\in C\left( [0,T],C^{1}_{b}(\Omega_{R})\right)^2}\left(\left\| g_{1}(U_{w})\right\| _{L^{1}\left( [0,T],L^{1}(\Omega_{R})\right)}\right)  \right)  \quad\text{and}\\
       		dN_{p}&:=\displaystyle \left\|  \dfrac{d \gamma_{p}}{d x}\right\| _{C\left([0,T]\times \Omega_{R}\right)}\left(\left\| P_{0}\right\| _{L^{1} (\Omega_{R})}+\sup_{w\in C\left( [0,T],C^{1}_{b}(\Omega_{R})\right)^2}\left( \left\| g_{1}(U_{w})\right\| _{L^{1}\left( [0,T],L^{1}(\Omega_{R})\right)} \right)\right)  , 
       	\end{align*}
       	and
       	\begin{align*}
       		N_{s}&:=\displaystyle \left\| \gamma_{s}\right\| _{C\left([0,T]\times \Omega_{R}\right)}\left(\left\| S_{0}\right\| _{L^{1} (\Omega_{R})}+\sup_{w\in C\left( [0,T],C^{1}_{b}(\Omega_{R})\right)^2}\left(\left\| g_{2}(U_{w})\right\| _{L^{1}\left( [0,T],L^{1}(\Omega_{R})\right)}\right)  \right),\\
       		dN_{s}&:=\displaystyle \left\|  \dfrac{d \gamma_{s}}{d x}\right\| _{C\left([0,T]\times \Omega_{R}\right)}\left(\left\| S_{0}\right\| _{L^{1} (\Omega_{R})}+\sup_{w\in C\left( [0,T],C^{1}_{b}(\Omega_{R})\right)^2}\left(\left\| g_{2}(U_{w})\right\| _{L^{1}\left( [0,T],L^{1}(\Omega_{R})\right)}\right)  \right) . 
       	\end{align*}
        Where $C\left( [0,T],C^{1}_{b}(\Omega_{R})\right)^2:=C\left( [0,T],C^{1}_{b}(\Omega_{R})\right)\times C\left( [0,T],C^{1}_{b}(\Omega_{R})\right)$.
       	\begin{rem}
       		We have shown in Theorem \ref{therom1} that by the compactness of initial data, the velocities, and the semilinear terms, we have the boundedness of the local solution, therefor $G(U_{w})=(g_{1}(U_{w}),g_{2}(U_{w}),g_{3}(U_{w}))$ is bounded for every $w\in C\left( [0,T],C^{1}_{b}(\Omega_{R})\right)^2 $.
       	\end{rem}
       	
       	We set $N=\max\{N_{p},N_{s}\}$ and $N_{z}=\max\{dN_{p},dN_{s}\}$. Then we consider 
       	\begin{equation}
       		B=\left\{f\in C\left( [0,T],C^{1}_{b}(\Omega_{R})\right)^2 \space | \enspace \|f\|_{C\left( [0,T],C(\Omega_{R})\right)^2}\leq N,\quad \|\nabla f\|_{C\left( [0,T],C(\Omega_{R})\right)^2}\leq N_{z}\right\}
       	\end{equation}
       	It is known that $B$ is closed subset of a Banach space. Now we check that $\mathcal{F}$ is a contraction on $B$.\\
       	Let $ w:=(w_{p},w_{s})\in B $, we consider the following mapping $\mathcal{F}$ as follow
      \begin{equation}
        \resizebox{\textwidth}{!}{$
       	\label{weightf2}
         \mathcal{F}\left( w\right) (t,x)=\left( \begin{array}{ll}
       \mathcal{F}_{1}\left( w\right) (t,x)\\
       	\mathcal{F}_{2}\left( w\right) (t,x)\\
       		\end{array}\right) =\left( \begin{array}{ll}
       			\displaystyle \int_{X_{w_{p}}[t,\Omega_{R}](0)}\gamma_{p}\left( x-X_{w_{p}}[0,y](t)\right) P_0(y)  \,d y+\int_{\Omega_{R}}\int_{0}^{T}\gamma_{p}\left(x- X_{w_{p}}[s,y](t)\right)g_{1}(U_{w})\left(s,y\right)\,dyds\\
       			\displaystyle \int_{X_{w_{s}}[t,\Omega_{R}](0)}\gamma_{s}\left( x-X_{w_{s}}[0,y](t)\right) S_0(y)  \,d y+\int_{\Omega_{R}}\int_{0}^{T}\gamma_{s}\left( x-X_{w_{s}}[s^{\prime},y](t)\right)g_{2}(U_{w})\left(s^{\prime},y\right)\,dyds^{\prime}\\
       		\end{array}
       		\right).$}
       	\end{equation}
        Let us introduce the following notations:
        $$\Lambda_1(P,g):= \left\|P\right\|_{L^1\left(\Omega_{R}\right)}+\sup_{w\in C\left( [0,T],C^{1}_{b}(\Omega_{R})\right)^2}\left(\left\| g(U_{w})\right\| _{L^{1}\left( [0,T], L^{1}(\Omega_{R})\right)}\right),$$
$$\Lambda_\infty(P,g):= \left\|P\right\|_{L^{\infty}(\Omega_{R})}+\sup_{w\in C\left( [0,T], C^{1}_{b}(\Omega_{R})\right)^2}\left(\left\| g(U_{w})\right\| _{L^{1}\left( [0,T], L^{\infty}(\Omega_{R})\right)}\right),$$
$$C_\infty(P,g,\gamma):= 1+8 R C_{v}\left\| \gamma\right\|_{C\left([0, t]\times \Omega_{R} \right)}\Lambda_{\infty}(P,g).$$
Following the same demonstration as done in \cite{keimer_existence_2018}, we can show that $\mathcal{F}(B)\subset B$ and we have for $w,w^{\prime}\in B$ the following estimates:
\begin{eqnarray*}
       		\left|\mathcal{F}_{1}(w)(t,x)- \mathcal{F}_{1}(w^{\prime})(t,x)\right| \leq 
       	 \left\| X_{w_{p}}[t, \cdot](*)-X_{w^{\prime}_{p}}[t, \cdot](*)\right\| _{C\left([0, t]\times \Omega_{R} \right)} \left\| \gamma_{p}(t,\cdot)\right\|_{C^{1}_{b}(\Omega_{R})} \Lambda_1(P_0,g_1)\; C_\infty(P_0,g_1,\gamma_p).
         \end{eqnarray*}
As
\begin{equation*}
     \left\| X_{w_{p}}[t, \cdot](*)-X_{w^{\prime}_{p}}[t, \cdot](*)\right\| _{C\left([0, t]\times \Omega_{R} \right)}\leq 
     T\|w_{p}-w_{p}^{\prime}\|_{C\left([0, t] ; C\left(\Omega_{R}\right)\right)} \exp\left( T\left(N_z+\|D\vec{\alpha_{p}}\|_{L^{\infty}([0,t]\times \Omega_{R})}\right)\right) ,
\end{equation*}
we get
\begin{eqnarray}\label{estimF1}
    \left|\mathcal{F}_{1}(w)(t,x)- \mathcal{F}_{1}(w^{\prime})(t,x)\right| \leq \hspace{11cm}\nonumber\\
    T\|w_{p}-w_{p}^{\prime}\|_{C\left([0, t] ; C\left(\Omega_{R}\right)\right)} \exp\left( T\left( N_z+\|D\vec{\alpha_{p}}\|_{L^{\infty}([0,t]\Omega_{R})}\right)\right)  \left\| \gamma_{p}(t,\cdot)\right\|_{C^{1}_{b}(\Omega_{R})}  \Lambda_1(P_0,g_1)
       		C_\infty(P_0,g_1,\gamma_p).
       	\end{eqnarray}
	Using the same estimates, we get similarly
\begin{eqnarray}\label{estimF2}
    \left|\mathcal{F}_{2}(w)(t,x)- \mathcal{F}_{2}(w^{\prime})(t,x)\right| \leq \hspace{11cm}\nonumber\\
    T\|w_{s}-w_{s}^{\prime}\|_{C\left([0, t] ; C\left(\Omega_{R}\right)\right)} \exp\left( T\left( N_z+\|D\vec{\alpha_{s}}\|_{L^{\infty}([0,t]\Omega_{R})}\right)\right)   \left\| \gamma_{s}(t,\cdot)\right\|_{C^{1}_{b}(\Omega_{R})}  \Lambda_1(S_0,g_2)
       		C_\infty(S_0,g_2,\gamma_s).
       	\end{eqnarray}

       	By summing the two estimates \eqref{estimF1} and \eqref{estimF2}, we deduce that $\mathcal{F}$ is Lipschitzien, then, by choosing $T$ small enough, we have
       	\begin{equation}
       		\left\|\mathcal{F}(w)-\mathcal{F}(w^{\prime}) \right\|_{C\left( [0,T],C(\Omega_{R})\right)^2} <\dfrac{1}{2} \left\| w-w^{\prime} \right\|_{C\left( [0,T],C(\Omega_{R})\right)^2}
       	\end{equation} 
       	Therefore, $\mathcal{F}$ is a contraction on $B$, then using Banach's fixed point theorem, there exists a unique fixed point $w\in B$ such that $\mathcal{F}(w)=w$ for $t \leq T^{*}$. Furthermore, \cite[Theorem 3.24]{keimer2018nonlocal} ensures that the existence and the uniqueness of the solution hold also for any final time $T>0$, hence, in our case we have the existence and uniqueness for all $t \leq T^{*}_2$ given in \eqref{defT2}.
       \end{proof}  
\section{Radial model reformulation} \label{sec3}
The experience shows that the tumor spheroid can be invariant by rotation and can guard the shape of the spheroid during the growth process in some cases. So,  we assume that the tumor remains in spheroid form for all $t\leq T$, and we use the assumption of rotational invariance to write the model in radial coordinates, which allows us to obtain a simplified expression of the model.\\
Let $f:\R^d \rightarrow \R$ a radial function, by definition, it exists a function  which we note $\widetilde{f}$ satisfies $$f(x)=\widetilde{f}(\|x\|)=\widetilde{f}(r)\qquad \text{for all}\quad x\in \R^d,\quad
\text{with}\quad r=\|x\|.$$
Under the assumption of invariance by rotation, $P$ is a radial function, we assume also that $S$, $M$,  $\vec{\alpha_k}$,$\gamma_k$, and $m$ are radial for $k=p,s$.\\
The next proposition is crucial for the characterisation of a radial expression for the nonlocal term.
\begin{prop}[The convolution of two radial functions]
	\label{convprop}
    Let $f$ and $g$ be two radial functions, defined from $\R^d$ to $\R$ such that $f*g$ is well defined. Then the convolution product $f*g$ is also a radial function. Furthermore, we have for all $x\in \R^d$:
    \begin{equation}
    	f*g(x)=\widetilde{f*g}(r)=\widetilde{f}\;\widetilde{*}\;\widetilde{g}(r),\\
      \end{equation}
      where $\widetilde{f}\;\widetilde{*}\;\widetilde{g}$ is defined as 
      \begin{equation}
\widetilde{f}\;\widetilde{*}\;\widetilde{g}(r)=\dfrac{2 \pi ^{\frac{d-1}{2}}}{\Gamma(\frac{d-1}{2})}\int^{+\infty}_{0}\left[\int^{\pi}_{0}\widetilde{f}\left( \sqrt{s^{2}+
      		r^{2}+2rs\cos(\theta)}\right) s^{d-1}\sin^{d-2}(\theta)d\theta\right]\widetilde{g}(s) ds ,
      \end{equation}
      where $\Gamma$ denotes  the classical Gamma function.
\end{prop}
\begin{proof}
  Let $x_{1}$ and x$_{2}$ two $\R^{d}$ vectors, such that $\|x_{1}\|=\|x_{2}\|$. Then it exists an orthogonal endomorphism $A$ such that $Ax_{2}=x_{1}$, then we have
$$f*g(x_{1})=\int_{\R^{d}}f(x_{1}-y)g(y)d y= \int_{\R^{d}}f(Ax_{2}-y)g(y)d y.$$
With the substitution $y=Az$, using the fact that $|\det(A)|=1$, we get
$$f*g(x_{1})=\int_{\R^{d}}f(Ax_{2}-Az)g(Az)d z.$$
So we have 
  \begin{equation*}	f*g(x_{1})=\int_{\R^{d}}f(A(x_{2}-z))g(Az)d z=\int_{\R^{d}}f(x_{2}-z)g(z)d z=f*g(x_{2}).
  \end{equation*}
Which shows that $f*g$ is a radial function. By using the fact that the convolution product $f*g$ is radial, we get
\begin{align*}
f*g(x)=f*g(\|x\|e_{1})&=\int_{\R^d}f(\|x\|e_{1}-y)g(y)d=\int_{\R^d}\widetilde{f}(\|\|x\|e_{1}-y\|)\widetilde{g}(\|y\|)d y\\
  	&=\int_{\R^{d} }\widetilde{f}\left( \sqrt{(\|x\|-y_{1})^{2}+\|\tilde{y}\|^2_2}\right)\widetilde{g}\left( \sqrt{y_{1}^{2}+\|\tilde{y}\|^2_2}\right) dy,
  	                   \end{align*}
with $e_1=(1,0,\dots,0)\in \R^d$ and $\tilde{y}=(y_{2},y_{3},\dots,y_{d})\in \R^{d-1}$.
  Using the Fubini theorem, we get the following:
  \begin{align*}
  	f*g(x) &=\int_{\R}\left( \int_{\R^{d-1}}\widetilde{f}\left( \sqrt{(\|x\|-y_{1})^{2}+\|\tilde{y}\|^2_2}\right)\widetilde{g}\left( \sqrt{y_{1}^{2}+\|\tilde{y}\|^2_2}\right)  d \tilde{y}\right) d y_{1}.
   \end{align*}
We recall that for a radial function $h$ on $\mathbb{R}^d$, we have the formula:
$$\int_{\mathbb{R}^d}h(x)dx=\omega_d\int_0^{+\infty}\widetilde{h}(r)r^{d-1}dr,$$
where $\omega_{d}=\dfrac{2 \pi ^{\frac{d}{2}}}{\Gamma(\frac{d}{2})}$ denotes  the measure of the unit sphere in $\R^{d}$. Then, we may write
\begin{align*} 
 f*g(x)  &=\omega_{d-1}\int_{\R}\left( \int_{0}^{+\infty}\widetilde{f}\left( \sqrt{(\|x\|-y_{1})^{2}+\tilde{r}^2}\right)\widetilde{g}\left( \sqrt{y_{1}^{2}+\tilde{r}^2}\right)\tilde{r}^{d-2}  d \tilde{r}\right) d y_{1},
\end{align*}
where $\tilde{r} =\|\tilde{y}\|$ is $2$-norm in $\mathbb{R}^{d-1}$. Using the polar coordinates:
  \begin{align*}
  	\Psi:] 0,+\infty[\times]-\frac{\pi}{2}, \frac{\pi}{2}[ & \longrightarrow] 0,+\infty[\times \mathbb{R} \\
  	(s, \varphi) & \longmapsto\left(\tilde{r}(s, \varphi)=s \cos (\varphi), y_{1}(s, \varphi)=s \sin (\varphi)\right),
  \end{align*}
we obtain
  \begin{align*}
  f*g(x)&=\omega_{d-1}\int_0^{+\infty}\int_{-\frac{\pi}{2}}^{\frac{\pi}{2}} \widetilde{f}\left(\sqrt{(\|x\|-s \sin (\varphi))^2+s^2 \cos ^2(\varphi)}\right) \widetilde{g}(s) s^{d-2}\cos^{d-2}(\varphi)  \,s d\varphi ds\\
  &=\omega_{d-1}\int_0^{+\infty}\int_{-\frac{\pi}{2}}^{\frac{\pi}{2}} \widetilde{f}\left(\sqrt{\|x\|^2-2 s\|x\|^2 \sin (\varphi)+s^2 }\right) \widetilde{g}(s) s^{d-1}\cos^{d-2}(\varphi)  d \varphi d s
\end{align*}
And by using the following change of variable $\theta=\varphi+\frac{\pi}{2}$, we get the desired result.
  \end{proof}

  Now we will transform the model (\ref{model}) into radial coordinates, due to the fact that the solution is compactly supported, see Theorem \ref{therom1}, then
we have
		\begin{align*}
			\operatorname{div}\left(\mathrm{V}[\vec{\alpha_{p}},\gamma_{p},P](t,x) P(t,x)\right)&=\operatorname{div}\left(\vec{\alpha_{p}}\left( \gamma_{p}*\left(\mathds{1}_{\Omega_{R}} P\right) \right) (t,x) P(t,x)\right)=\operatorname{div}\left(\vec{\alpha_{p}}\left(\gamma_{p}*P\right)(t,x) P(t,x)\right)\\
			&= \operatorname{div}\left(\vec{\alpha_{p}}(t,x)\left(\gamma_{p}*P\right)(t,x)\right)P(t,x)+\left(\vec{\alpha_{p}}(t,x)\left(\gamma_{p}*P\right) (t,x)\right)\cdot\nabla P(t,x)\\
			&= \Bigl[\operatorname{div}\left(\vec{\alpha_{p}}(t,x)\right) \left(\gamma_{p}*P\right)(t,x)+\vec{\alpha_{p}}(t,x)\cdot\nabla \left( \left(\gamma_{p}*P\Bigr)(t,x)\right)\right]P(t,x)+\\
			&+\left(\vec{\alpha_{p}}(t,x)\left(\gamma_{p}*P\right) (t,x)\right)\cdot\nabla P(t,x),
		\end{align*}	The directional vector of the velocity $\vec{\alpha_{p}}$, is assumed to be a radial vector, so it can be written as follows
		\begin{equation*}
			\vec{\alpha_{p}}(t,x)=\widetilde{\alpha_{p}}(t,r)\vec{e_{r}},
			\end{equation*}
				with $\vec{e}_{r}= \dfrac{x}{\|x\|}$. So we have
\begin{align*}
\operatorname{div}\left(\vec{\alpha_{p}}(t,x)\right)=\dfrac{d-1}{r }\widetilde{\alpha_{p}}(t,r)+\dfrac{\partial \widetilde{\alpha_{p}}(t,r) }{\partial r}=\dfrac{1}{r^{d-1}}\dfrac{\partial}{\partial r}\left(r^{d-1}\widetilde{\alpha_{p}}(t,r) \right),
\end{align*}
and as 
\begin{equation*}
    \nabla \left( \left( \gamma_{p}* P\right) (t,x)\right)=\dfrac{\partial ( \widetilde{\gamma_{p}*P})}{\partial r} (t,r)\vec{e_r},\quad\text{and}\quad
    \nabla P(t,x)=\dfrac{\partial \widetilde{P}}{\partial r}(t,r)\vec{e_r},
\end{equation*}
we get the following:
				\begin{equation*}
					\vec{\alpha_{p}}(t,x)\cdot\nabla \left( \left( \gamma_{p}* P\right) (t,x)\right)=\widetilde{\alpha_{p}}(t,r)\dfrac{\partial ( \widetilde{\gamma_{p}*P})}{\partial r} (t,r)=\widetilde{\alpha_{p}}(t,r)\dfrac{\partial \widetilde{\gamma_{p}}}{\partial r}\; \widetilde{*}\; \widetilde{P}(t,r),
				\end{equation*}
				and
				\begin{equation*}
					\vec{\alpha_{p}}(t,x)\left(\gamma_{p}*P\right) (t,x)\cdot\nabla P(t,x)=\widetilde{\alpha_{p}}(t,r)(\widetilde{\gamma_{p}*P})(t,r)\dfrac{\partial \widetilde{P}}{\partial r}(t,r) .
				\end{equation*}
				Thus, we have
				\begin{align*}			\operatorname{div}\left(\mathrm{V}[\vec{\alpha_{p}},\gamma_{p},P](t,x) P(t,x)\right)=&\Bigl[
   \dfrac{1}{r^{d-1}}\dfrac{\partial}{\partial r}\Bigl(r^{d-1}\widetilde{\alpha_{p}}(t,r) \Bigr)
(\widetilde{\gamma_{p}}\;\widetilde{*}\;\widetilde{P})(t,r)+\widetilde{\alpha_{p}}(t,r)( \dfrac{\partial \widetilde{\gamma_{p}}}{\partial r} \;\widetilde{*}\; \widetilde{P}) (t,r)\Bigr]\widetilde{P}(t,r)\\
					&+
					\widetilde{\alpha_{p}}(t,r)(\widetilde{\gamma_{p}}\;\widetilde{*}\;\widetilde{P})(t,r)\dfrac{\partial \widetilde{P}}{\partial r}(t,r)\\
     &=\widetilde{\alpha_{p}}(t,r)\dfrac{\partial ((\widetilde{\gamma_{p}}\;\widetilde{*}\;\widetilde{P})\widetilde{P})}{\partial r}(t,r)+\dfrac{1}{r^{d-1}}\dfrac{\partial}{\partial r}\Bigl(r^{d-1}\widetilde{\alpha_{p}}(t,r) \Bigr)
(\widetilde{\gamma_{p}}\;\widetilde{*}\;\widetilde{P})(t,r)\widetilde{P}(t,r).
				\end{align*}    
     Finally we can write the equation for proliferation cells in the system \eqref{model} in the following form
				\begin{equation}
    \resizebox{0.94\textwidth}{!}{$
				\dfrac{\partial \widetilde{P}}{\partial t}(t,r)+\widetilde{\alpha_{p}}(t,r)\dfrac{\partial ((\widetilde{\gamma_{p}}\;\widetilde{*}\;\widetilde{P})\widetilde{P})}{\partial r}(t,r) 
					=\left(\widetilde{m}(r)H(\widetilde{M}(t,r))-a_{1}\widetilde{m}(r)\lambda(t)\widetilde{M}(t,r)-\dfrac{1}{r^{d-1}}\dfrac{\partial}{\partial r}\Bigl(r^{d-1}\widetilde{\alpha_{p}}(t,r) \Bigr)(\widetilde{\gamma_{p}}\;\widetilde{*}\;\widetilde{P})(t,r)\right)\widetilde{P}(t,r).$}
       	\end{equation}
    In the same way, for the equation of healthy cells, we have
    \begin{equation}
    \resizebox{0.94\textwidth}{!}{$
				 \dfrac{\partial \widetilde{S}}{\partial t}(t,r)+\widetilde{\alpha_{s}}(t,r)\dfrac{\partial ((\widetilde{\gamma_{s}}\;\widetilde{*}\;\widetilde{S})\widetilde{S})}{\partial r}(t,r) 
					=\left(-a_2\widetilde{m}(r)\lambda(t)\widetilde{M}(t,r)-\dfrac{1}{r^{d-1}}\dfrac{\partial}{\partial r}\Bigl(r^{d-1}\widetilde{\alpha_{s}}(t,r) \Bigr)(\widetilde{\gamma_{s}}\;\widetilde{*}\;\widetilde{S})(t,r)\right)\widetilde{S}(t,r).$}
 				\end{equation}
We recall that 
    \begin{equation}
       \Delta M(t,x)=\Delta \widetilde{M}\left(t,\|x\|_2\right)
        =\frac{d-1}{r}\frac{\partial \widetilde{M}}{\partial r}(t,r)+\frac{\partial^{2}\widetilde{M}}{\partial r^2}(t,r)=\dfrac{1}{r^{d-1}}\dfrac{\partial}{\partial r}\left(r^{d-1} \dfrac{\partial \widetilde{M}}{\partial r}(t,r)\right).
     \end{equation}
					Finally, we write the system (\ref{model}) in radial coordinates as follows     
			\begin{equation}
   \resizebox{\textwidth}{!}{$
				\label{radialmodel}
				\begin{cases}
					\vspace{0.2cm}
					\dfrac{\partial \widetilde{P}}{\partial t}(t,r)+\widetilde{\alpha_{p}}(t,r)\dfrac{\partial ((\widetilde{\gamma_{p}}\;\widetilde{*}\;\widetilde{P})\widetilde{P})}{\partial r}(t,r) 
					=\left(\widetilde{m}(r)H(\widetilde{M}(t,r))-a_{1}\widetilde{m}(r)\lambda(t)\widetilde{M}(t,r)-\dfrac{1}{r^{d-1}}\dfrac{\partial}{\partial r}\Bigl(r^{d-1}\widetilde{\alpha_{p}}(t,r) \Bigr)(\widetilde{\gamma_{p}}\;\widetilde{*}\;\widetilde{P})(t,r)\right)\widetilde{P}(t,r), \\
					\vspace{0.2cm}
					\dfrac{\partial \widetilde{S}}{\partial t}(t,r)+\widetilde{\alpha_{s}}(t,r)\dfrac{\partial ((\widetilde{\gamma_{s}}\;\widetilde{*}\;\widetilde{S})\widetilde{S})}{\partial r}(t,r) 
					=\left(-a_2\widetilde{m}(r)\lambda(t)\widetilde{M}(t,r)-\dfrac{1}{r^{d-1}}\dfrac{\partial}{\partial r}\Bigl(r^{d-1}\widetilde{\alpha_{s}}(t,r) \Bigr)(\widetilde{\gamma_{s}}\;\widetilde{*}\;\widetilde{S})(t,r)\right)\widetilde{S}(t,r),\quad(t,r)\in [0,T]\times ]0,R]\\
					\vspace{0.2cm}
					\dfrac{\partial \widetilde{M}}{\partial t}(t,r) -\dfrac{D}{r^{d-1}}\dfrac{\partial}{\partial r}\left(r^{d-1} \dfrac{\partial \widetilde{M}}{\partial r}(t,r)\right)=M_{s}\widetilde{S}(t,r)\left(1-\widetilde{M}(t,r)\right)-\eta \widetilde{M}(t,r) \widetilde{P}(t,r), \hspace{3cm}(t,r)\in [0,T]\times ]0,R],\\
			  \vspace{0.2cm}
         P(t,R)=0, \quad S(t,R)=S_R, \quad M(t,R)=M_R, \hspace{8.15cm} t\in [0,T], \\
         \vspace{0.2cm}
                      P(0,r)=P_{0}(r), \quad
                      S(0,r)=S_{0}(r),\quad
                      M(0,r)=M_{0}(r),\hspace{7cm} r \in ]0,R].
				 \end{cases}
    $}                  
			\end{equation}
   We denote by $\widetilde{U}=\left(\widetilde{P},\widetilde{S},\widetilde{M}\right)$ the exact solution of the radial problem \eqref{radialmodel}. Then 
 the system \eqref{radialmodel} may be written in the following form:  
       \begin{equation}
       	\label{equationURadialUTilde}
       	\left\{
       	\begin{array}{ccc}
       		\dfrac{\partial \widetilde{U}}{\partial t}(t,r)+\mathcal{\widetilde{A}}\left(\widetilde{U}\right)(t,r) =\widetilde{G}(\widetilde{U})(t,r), \quad &(t,r)\in[0,T]\times  ]0,R],\\ 
       		\widetilde{U}(t,R) =\widetilde{U}_{R}, \quad &t\in[0,T],\\ 
       		\widetilde{U}(0,r)=\widetilde{U}_{0}(r). \quad & r\in  ]0,R],
       	\end{array}
       	\right.
       \end{equation}
with 
\begin{equation*}
\resizebox{0.85\textwidth}{!}{$
\displaystyle
\mathcal{\widetilde{A}}\left(\widetilde{U}\right)(t,r)=\left[\widetilde{\alpha_{p}}(t,r)\dfrac{\partial ((\widetilde{\gamma_{p}}\;\widetilde{*}\;\widetilde{P})\widetilde{P})}{\partial r}(t,r) ,\widetilde{\alpha_{s}}(t,r)\dfrac{\partial ((\widetilde{\gamma_{s}}\;\widetilde{*}\;\widetilde{S})\widetilde{S})}{\partial r}(t,r) ,-\dfrac{D}{r^{d-1}}\dfrac{\partial}{\partial r}\left(r^{d-1} \dfrac{\partial \widetilde{M}}{\partial r}(t,r)\right)\right]$},
\end{equation*}
\begin{equation*}
       	\widetilde{G}(\widetilde{U})(t,r)=\Bigr(\widetilde{g}_{1}(\widetilde{U})(t,r),\widetilde{g}_{2}(\widetilde{U})(t,r),\widetilde{g}_{3}(\widetilde{U})(t,r) \Bigr),
 \end{equation*}       
  where  
  $$
\begin{array}{rcl}
\widetilde{g}_{1}(\widetilde{U})(t,r)&=&\left(\widetilde{m}(r)H(\widetilde{M}(t,r))-a_{1}\widetilde{m}(r)\lambda(t)\widetilde{M}(t,r)-\dfrac{1}{r^{d-1}}\dfrac{\partial}{\partial r}\Bigl(r^{d-1}\widetilde{\alpha_{p}}(t,r) \Bigr)(\widetilde{\gamma_{p}}\;\widetilde{*}\;\widetilde{P})(t,r)\right)\widetilde{P}(t,r), ),\\
\widetilde{g}_{2}(\widetilde{U})(t,r)&=&\left(-a_2\widetilde{m}(r)\lambda(t)\widetilde{M}(t,r)-\dfrac{1}{r^{d-1}}\dfrac{\partial}{\partial r}\Bigl(r^{d-1}\widetilde{\alpha_{s}}(t,r) \Bigr)(\widetilde{\gamma_{s}}\;\widetilde{*}\;\widetilde{S})(t,r)\right)\widetilde{S}(t,r),
\\
\widetilde{g}_{3}(\widetilde{U})(t,r)&=&\dfrac{\partial \widetilde{M}}{\partial t}(t,r) -\dfrac{D}{r^{d-1}}\dfrac{\partial}{\partial r}\left(r^{d-1} \dfrac{\partial \widetilde{M}}{\partial r}(t,r)\right)=M_{s}\widetilde{S}(t,r)\left(1-\widetilde{M}(t,r)\right)-\eta \widetilde{M}(t,r) \widetilde{P}(t,r),
\end{array}
$$

$U_{R}=(0,S_{R},M_{R})$ and $U_{0}=(P_0,S_{0},M_{0})$.

\section{The BDF B-spline method}
 	\label{sec4}
In this section, we provide a description of the collocation method used to solve the radial system \eqref{equationURadialUTilde}. The collocation method utilizes the cubic B-splines, which are widely favored in numerical analysis due to their inherent smoothness and flexibility in function approximation. B-splines are piecewise polynomial functions that ensure continuity up to the second derivative, making them particularly effective for interpolating data points or solving differential equations.

To begin with, we uniformly divide the spatial interval $[0, R]$ into $N+7$ points. This division is crucial for applying the B-spline method as it specifies the nodes where the B-splines will be computed. The subdivision points are labeled as follows:
\begin{equation*}
\label{subdivsion}
      r_{-3} < r_{-2} < r_{-1} < 0 < r_{0} < r_{1} < \dots < r_{N} = R < r_{N+1} < r_{N+2} < r_{N+3},
\end{equation*}
where $[r_0,R]$ is our domain of study, with $r_0$ is a very small value, and the nodes $r_{i}$ are defined by the relation $r_{i} = ih$ for $i = -3, \dots ,N+3$ with the step size $h$ given by $ h = \dfrac{R-r_0}{N}.$
The points $r_{-3}, r_{-2}, r_{-1}, r_{N+1}, r_{N+2}, r_{N+3}$ are added to manage boundary conditions and ensure the spline approximation remains accurate and stable near the edges of the interval.
In this context, the fundamental $B$-spline function refers to the cubic spline centered at the nodes $-2, -1, 0, 1, 2$, and its support is limited to the interval $[-2, 2]$. It can be written as follows:
\begin{equation}
\label{Bsplines}
B(r)=\left\{\begin{array}{ccc}
\vspace{0.4cm}
0 & \text { if } & r<-2 \text { or } r \geq 2, \\
\vspace{0.4cm}
\frac{1}{6}(2+r)^3 & \text { if } & -2 \leq r<-1, \\
\vspace{0.4cm}
\frac{1}{6}\left(4-6 r^2-3 r^3\right) & \text{ if } & -1 \leq r<0, \\
\vspace{0.4cm}
\frac{1}{6}\left(4-6 r^2+3 r^3\right) & \text{ if } & 0 \leq r<1, \\

\frac{1}{6}(2-r)^3 & \text { if } & 1 \leq r<2 ,
\end{array}\right.
\end{equation}
and $B_i$ represent the well-known $B$-spline functions associated to the nodes $r_i$ defined for $i=-1, \ldots, N+1$ by
$$
B_i(r)=B\left(\frac{r-r_i}{h}\right) .
$$
The function $B_i$ is supported by the interval $\left[r_{i-2}, r_{i+2}\right]$. Table \ref{table:bsplines} provides a summary of the values of $B_i$ and its derivatives at the points $r_i$.
\begin{table}[H]
  \centering
  \begin{tabular}{|c||ccccc|}
\hline$x$ & $r_{i-2}$ & $r_{i-1}$ & $r_i$ & $r_{i+1}$ & $r_{i+2}$ \\
\hline$B_i(r)$ & 0 & $1 / 6$ & $4 / 6$ & $1 / 6$ & 0 \\
\hline$B_i^{\prime}(r)$ & 0 & $-1 / 2 h$ & 0 & $1 / 2 h$ & 0 \\
\hline$B_i^{\prime \prime}(r)$ & 0 & $1 / h^2$ & $-2 / h^2$ & $1 / h^2$ & 0 \\
\hline
\end{tabular}
    \caption{Values of B-splines and their derivatives at the points $r_{i}$.}
    \label{table:bsplines}
\end{table}
The exact solution of (\ref{equationURadialUTilde} is approximated by a cubic B-spline given by the following expression:
 \begin{equation*}
     \widetilde{U}_{h}(t,r)=\left(\widetilde{P}_{h}(t,r),\widetilde{S}_{h}(t,r),\widetilde{M}_{h}(t,r)\right)=\left(\sum_{i=-1}^{N+1} \alpha_i^{(1)}(t) B_i(r),\sum_{i=-1}^{N+1} \alpha_i^{(2)}(t) B_i(r),\sum_{i=-1}^{N+1} \alpha_i^{(3)}(t) B_i(r)\right)
\end{equation*}
where $\alpha_i^{(k)}(t)$ are time-dependent coefficients with unknown values for $k=1,2,3$. 
Consider the following vector-valued functions of dimensions $(N-1) \times 1$ and $3(N-1) \times 1$ respectively:
\begin{equation}
	\mathds{B}(r)=\left(\begin{array}{c}
	B_1(r) \\
	\vdots \\
	B_{N-1}(r)
\end{array}\right) \quad \text { and } \quad \mathds{\upalpha}(t)=\left[\begin{array}{c}
	\alpha^{(1)}(t) \\
	\alpha^{(2)}(t) \\
	\alpha^{(3)}(t)
\end{array}\right]=\left(\begin{array}{c}
	\alpha_{1}^{(1)}(t) \\
	\vdots \\
	\alpha_{N-1}^{(1)}(t)\\
	\alpha_{1}^{(2)}(t) \\
	\vdots \\
	\alpha_{N-1}^{(2)}(t)\\
	\alpha_{1}^{(3)}(t) \\
	\vdots \\
	\alpha_{N-1}^{(3)}(t)
\end{array}\right)
\end{equation}
The function $U_{h}(t,x)$ has the following form
\begin{equation}
	\label{Uh}
	\widetilde{U_h}(t,r)=\left( \begin{array}{ll}
		\widetilde{P_{h}}(t,r)\\
		\widetilde{S_{h}}(t,r)\\
		\widetilde{M_{h}}(t,r)
	\end{array}\right)  =\left( \begin{array}{ll}
	\alpha_{-1}^{(1)}(t) B_{-1}(r)+\alpha_0^{(1)}(t) B_0(r)+\mathds{B}(r)^T \alpha^{(1)}(t)+\alpha_N^{(1)}(t) B_N(r)+\alpha_{N+1}^{(1)}(t) B_{N+1}(r)\\
	\alpha_{-1}^{(2)}(t) B_{-1}(r)+\alpha_0^{(2)}(t) B_0(r)+\mathds{B}(r)^T \alpha^{(2)}(t)+\alpha_N(t)^{(2)} B_N(r)+\alpha_{N+1}^{(2)}(t) B_{N+1}(r)\\
	\alpha_{-1}^{(3)}(t) B_{-1}(r)+\alpha_0^{(3)}(t) B_0(r)+\mathds{B}(r)^T \alpha^{(3)}(t)+\alpha_N(t)^{(3)} B_N(r)+\alpha_{N+1}^{(3)}(t) B_{N+1}(r)
	\end{array}\right),
\end{equation}
where the notation $\mathds{B}(r)^T$ stands for the transpose of the vector $\mathds{B}(r)$.
The collocation method involves substituting $\widetilde{U_{h}}$ and its derivatives in \eqref{equationURadialUTilde} with the expression for $\widetilde{U_{h}}$ given by (\ref{Uh}). This substitution is followed by evaluating the resulting equation at the points $r_{i}$ for $i = 0, \dots,N$. Consequently, we obtain an ordinary differential equation that we can solve later. Beginning with the first equation for tumor cells $\widetilde{P_{h}}$, we have:
  \begin{equation*}
      \dfrac{\partial \widetilde{P}_{h}}{\partial t}(t,r)+\widetilde{\alpha_{p}}(t,r)\dfrac{\partial \left( (\widetilde{\gamma_{p}}\;\widetilde{*}\;\widetilde{P}_{h})\widetilde{P}_{h}\right) }{\partial r}(t,r) =\widetilde{g}_1(\widetilde{U}_h)(t,r),
  \end{equation*}
 	which may be written in the form
  \begin{multline}
 		\dfrac{\partial \widetilde{P}_{h}}{\partial t}(t,r)+\left(\widetilde{\gamma_{p}}\;\widetilde{*}\;\widetilde{P}_{h}\right)(t,r)\left[ \dfrac{1}{r^{d-1}}\dfrac{\partial}{\partial r}\Bigl(r^{d-1}\widetilde{\alpha_{p}}(t,r) \Bigr)\widetilde{P}_{h}(t,r)+\widetilde{\alpha_{p}}(t,r)\dfrac{\partial \widetilde{P}_{h}}{\partial r}(t,r)\right] +\left(\dfrac{\partial \widetilde{\gamma_{p}}}{\partial r} \widetilde{*} \widetilde{P}_{h}\right) (t,r)\widetilde{P}_{h}(t,r) =\\
 	=\left(H(\widetilde{M}_{h}(t,r))-a_{1}\lambda(t)\widetilde{M}_{h}(t,r)\right)\widetilde{P}_{h}(t,r) ,
 \end{multline}
 We first start by expressing the nonlocal term,then
 \begin{equation}
 	\label{Papp}
 	\left( \widetilde{\gamma_{p}}\;\widetilde{*}\;\widetilde{P}_{h}\right) (t,r)=\left( \widetilde{\gamma_{p}}\;\widetilde{*}\;\left( \sum_{i=-1}^{N+1} \alpha_i^{(1)}(t) B_i\right) \right) (t,r)=\sum_{i=-1}^{N+1}\alpha_i^{(1)}(t)\left( \widetilde{\gamma_{p}}\;\widetilde{*}\;   B_i \right) (r)
 \end{equation} 
 We note for $i=-1,\dots,N+1$, and $k=p,n$:
 \begin{equation}
   \mathrm{Z}^{k}_{i}(r):= \left(\widetilde{\gamma_{p}}\;\widetilde{*}\;   B_i\right)(r), \quad\text{and}\qquad
 		\overline{\mathrm{Z}}^{k}_{i}(r):= \dfrac{\partial \left(\widetilde{\gamma_{p}}\;\widetilde{*}\;  B_i\right)(r)}{\partial r}=	 \left(\dfrac{\partial\widetilde{\gamma_{p}}}{\partial r}\;\widetilde{*}\;  B_i\right)(r).
\end{equation}
 Then by substituting in (\ref{Papp}), we get
 
 \begin{multline}
 	\label{Pa}
 	\resizebox{0.95\textwidth}{!}{$\displaystyle \sum_{i=-1}^{N+1} \frac{d \alpha_i^{(1)}}{d t}(t)B_i(r)+\sum_{i=-1}^{N+1}\alpha_{i}^{(1)}(t)\mathrm{Z}^{p}_{i}(r)\left[ \dfrac{1}{r^{d-1}}\dfrac{\partial}{\partial r}\Bigl(r^{d-1}\widetilde{\alpha_{p}}(t,r) \Bigr)\sum_{i=-1}^{N+1}\alpha_{i}^{(1)}(t)B_{i}(r) +\widetilde{\alpha_{p}}(t,r)\sum_{i=-1}^{N+1} \alpha_i^{(1)}(t) B_i^{\prime}(r)\right]$}\\
 	\resizebox{0.95\textwidth}{!}{$\displaystyle +\widetilde{\alpha_{p}}(t,r)\sum_{i=-1}^{N+1} \alpha_i^{(1)}(t)\overline{\mathrm{Z}}^{p}_{i}(r)\sum_{i=-1}^{N+1}\alpha_{i}^{(1)}(t)B_{i}(r)
 	=\left(H\left( \sum_{i=-1}^{N+1}\alpha_{i}^{(3)}(t)B_{i}(r)\right) -a_{1}\lambda(t)\sum_{i=-1}^{N+1}\alpha_{i}^{(3)}(t)B_{i}(r)\right)\sum_{i=-1}^{N+1}\alpha_{i}^{(1)}(t)B_{i}(r) .$}
 \end{multline}
 For healthy cells $\widetilde{S_{h}}$, similar calculations lead to
\begin{multline}
	\label{Sa}
 	\sum_{i=-1}^{N+1} \frac{d \alpha_i^{(2)}}{d t}(t)B_i(r)+\sum_{i=-1}^{N+1}\alpha_{i}^{(2)}(t)\mathrm{Z}^{s}_{i}(r)\left[ \dfrac{1}{r^{d-1}}\dfrac{\partial}{\partial r}\Bigl(r^{d-1}\widetilde{\alpha_{p}}(t,r) \Bigr)\sum_{i=-1}^{N+1}\alpha_{i}^{(2)}(t)B_{i}(r) +\widetilde{\alpha_{s}}(t,r)\sum_{i=-1}^{N+1} \alpha_i^{(2)}(t) B^{\prime}_i(r)\right]\\
+\widetilde{\alpha_{s}}(t,r)\sum_{i=-1}^{N+1} \alpha_i^{(2)}(t)\overline{\mathrm{Z}}^{s}_{i}(r)\sum_{i=-1}^{N+1}\alpha_{i}^{(2)}(t)B_{i}(r)
= -a_{2}\lambda(t)\sum_{i=-1}^{N+1}\alpha_{i}^{(3)}(t)B_{i}(r)\sum_{i=-1}^{N+1}\alpha_{i}^{(2)}(t)B_{i}(r) .
\end{multline}
Similarly, for the third equation describing the concentration of oxygen and nutrients in the tissue, we have  
\begin{multline}
	\label{Ma}
	\sum_{i=-1}^{N+1} \frac{d \alpha_i^{(3)}}{d t}(t)B_i(r)-D\left( \frac{d-1}{r}\sum_{i=-1}^{N+1} \alpha_i^{(3)}(t)  B_i^{\prime}(r)+\sum_{i=-1}^{N+1}\alpha_i^{(3)}(t) B_i^{\prime\prime}(r)\right) =\\
	=M_{s}\sum_{i=-1}^{N+1} \alpha_i^{(2)}(t) B_i(r)\left( 1-\sum_{i=-1}^{N+1} \alpha_i^{(3)}(t) B_i(r)\right) -\eta \sum_{i=-1}^{N+1} \alpha_i^{(3)}(t) B_i(r) \sum_{i=-1}^{N+1} \alpha_i^{(1)}(t) B_i(r).
\end{multline}

To determine the values of $\alpha$ at the boundaries, we consider $\widetilde{U}_{h} = \left( \widetilde{P}_{h}, \widetilde{S}_{h}, \widetilde{M}_{h} \right)^{T}$ as natural cubic splines. Natural cubic splines require that their second derivatives vanish at the endpoints of the interval $[0, R]$. Hence, we have:

\begin{equation}
	\label{dUh}
	\dfrac{\partial^{2} \widetilde{U_h}}{\partial r^{2}}(t,r_{0})=\left( \begin{array}{ll}
		\vspace{0.15cm}
		\dfrac{\partial^{2} \widetilde{P_h}}{\partial r^{2}}(t,r_{0})\\
		\vspace{0.15cm}
		\dfrac{\partial^{2} \widetilde{S_h}}{\partial r^{2}}(t,r_{0})\\
		\dfrac{\partial^{2} \widetilde{M_h}}{\partial r^{2}}(t,r_{0})
	\end{array}\right)  =\left( \begin{array}{ll}
	\vspace{0.15cm}
		\dfrac{1}{h^2} \alpha_{-1}^{(1)}(t)-\dfrac{2}{h^2}
		 \alpha_0^{(1)}(t)+\dfrac{1}{h^2} \alpha_1^{(1)}(t)\\
		 \vspace{0.15cm}
		\dfrac{1}{h^2} \alpha_{-1}^{(2)}(t)-\dfrac{2}{h^2} \alpha_0^{(2)}(t)+\dfrac{1}{h^2} \alpha_1^{(2)}(t)\\
		\dfrac{1}{h^2} \alpha_{-1}^{(3)}(t)-\dfrac{2}{h^2} \alpha_0^{(3)}(t)+\dfrac{1}{h^2} \alpha_1^{(3)}(t)
	\end{array}\right)=\left( \begin{array}{ll}
	\vspace{0.15cm} 0 ,\\ 
		\vspace{0.15cm} 0,\\ 0. \end{array}\right)  
\end{equation}
Which leads to 
\begin{equation}
	\label{naturalr0}
	\alpha_{-1}^{j}(t)=2\alpha_{0}^{j}(t)-\alpha_{1}^{j}(t),\qquad \text{for}\quad j=1,2,3.
\end{equation} 
On the other hand, we set the value of the solution in the point $r=r_{0}$ to be fixed and note it $U_{h}(t,r_{0}):=\left( P_{L}(t),S_{L}(t),M_{L}(t)\right)^{T}$. Then we have 

\begin{equation}
	\label{Uh3}
	\widetilde{U_h}(t,r_{0})=\left( \begin{array}{ll}
		\vspace{0.15cm}
		\widetilde{P_{h}}(t,r_{0})\\
		\vspace{0.15cm}
		\widetilde{S_{h}}(t,r_{0})\\
		\widetilde{M_{h}}(t,r_{0})
	\end{array}\right)  =\left( \begin{array}{ll}
	\vspace{0.15cm}
		\dfrac{1}{6}\left(\alpha_{-1}^{(1)}(t)+4\alpha_0^{(1)}(t) +\alpha_{1}^{(1)}(t) \right)\\
		\vspace{0.15cm}
		\dfrac{1}{6}\left(\alpha_{-1}^{(2)}(t)+4\alpha_0^{(2)}(t) +\alpha_{1}^{(2)}(t) \right)\\
		\dfrac{1}{6}\left(\alpha_{-1}^{(3)}(t)+4\alpha_0^{(3)}(t) +\alpha_{1}^{(3)}(t) \right)
	\end{array}\right)=  \left( \begin{array}{ll}
	\vspace{0.15cm}
	\alpha_0^{(1)}(t) \\
	\vspace{0.15cm}
	\alpha_0^{(2)}(t)\\
	\alpha_0^{(3)}(t)
	\end{array}\right)=\left( \begin{array}{ll}
	\vspace{0.15cm}
	P_{L}(t)\\
	\vspace{0.15cm}
	S_{L}(t)\\
	M_{L}(t)
	\end{array}\right)
\end{equation}
 By an analogue observations, we get for $r=r_{N+1}=R$ the following
\begin{equation}
	\label{naturalR}
	\alpha_{N+1}^{j}(t)=2\alpha_{N}^{j}(t)-\alpha_{N-1}^{j}(t),\qquad \text{for}\quad j=1,2,3,
\end{equation} 
and 
\begin{equation}
	\label{alphaR}
	 \left( \begin{array}{ll}
	\vspace{0.15cm}
	\alpha_{N+1}^{(1)}(t) \\
	\vspace{0.15cm}
	\alpha_{N+1}^{(2)}(t)\\
	\alpha_{N+1}^{(3)}(t)
	\end{array}\right)=\left( \begin{array}{ll}
	\vspace{0.15cm}
	P_{R},\\
	\vspace{0.15cm}
	S_{R},\\
	M_{R}.
	\end{array}\right)
\end{equation}

   Let note for $k=p,s$ the following vector valued functions of size $N-1$ as follows:
   \begin{equation}
   \mathds{Z}_{k}(r):=\left(\begin{array}{c}
   	\mathrm{Z}^{k}_{1}(r)-\mathrm{Z}^{k}_{-1}(r)\\
   	\mathrm{Z}^{k}_{2}(r)\\
   	\vdots \\
   	\mathrm{Z}^{k}_{N-2}(r)\\
   	\mathrm{Z}^{k}_{N-1}(r)-\mathrm{Z}^{k}_{N+1}(r)
   \end{array}\right), \qquad\text{and}\qquad \overline{\mathds{Z}}_{k}(r):=\left(\begin{array}{c}
   \overline{\mathrm{Z}}^{k}_{1}(r)-\overline{\mathrm{Z}}^{k}_{-1}(r)\\
   \overline{\mathrm{Z}}^{k}_{2}(r)\\
   \vdots \\
   \overline{\mathrm{Z}}^{k}_{N-2}(r)\\
  \overline{\mathrm{Z}}^{k}_{N-1}(r)-\overline{\mathrm{Z}}^{k}_{N+1}(r)
   \end{array}\right).
   \end{equation}
   Furthermore, we define the following functions
   \begin{align}
   \mathrm{Q}_{p}(r)&:=P_{L}(t)\left(2\mathrm{Z}^{p}_{-1}(r)+ \mathrm{Z}^{p}_{0}(r)\right) +P_{R}\left( 2\mathrm{Z}^{p}_{N+1}(r)+\mathrm{Z}^{p}_{N+1}(r)\right),\\
   \overline{\mathrm{Q}}_{p}(r)&:=P_{L}(t)\left(2\overline{\mathrm{Z}}^{p}_{-1}(r)+ \overline{\mathrm{Z}}^{p}_{0}(r)\right) +P_{R}\left( 2\overline{\mathrm{Z}}^{p}_{N+1}(r)+\overline{\mathrm{Z}}^{p}_{N+1}(r)\right),\\
    \mathrm{Q}_{s}(r)&:=P_{L}(t)\left(2\mathrm{Z}^{s}_{-1}(r)+ \mathrm{Z}^{s}_{0}(r)\right) +S_{R}\left( 2\mathrm{Z}^{s}_{N+1}(r)+\mathrm{Z}^{s}_{N+1}(r)\right),\\
   \overline{\mathrm{Q}}_{s}(r)&:=S_{L}(t)\left(2\overline{\mathrm{Z}}^{s}_{-1}(r)+ \overline{\mathrm{Z}}^{s}_{0}(r)\right) +S_{R}\left( 2\overline{\mathrm{Z}}^{s}_{N+1}(r)+\overline{\mathrm{Z}}^{s}_{N+1}(r)\right).
   \end{align}
   Then, by using the relations (\ref{naturalr0}), (\ref{Uh3}), (\ref{naturalR}) and (\ref{alphaR}), we can express the nonlocal terms as follows 
   \begin{align*}
   	\sum_{i=-1}^{N+1}\alpha_{i}^{(1)}(t)\mathrm{Z}^{p}_{i}(r)&=\alpha_{-1}^{(1)}(t) \mathrm{Z}^{p}_{-1}(r)+\alpha_0^{(1)}(t)\mathrm{Z}^{p}_{0}(r)+\sum_{i=1}^{N-1}\alpha_{i}^{(1)}(t)\mathrm{Z}^{p}_{i}(r)+\alpha_N^{(1)}(t) \mathrm{Z}^{p}_{N}(r)+\alpha_{N+1}^{(1)}(t)\mathrm{Z}^{p}_{N+1}(r),\\
   	&=P_{L}(t)\left(2\mathrm{Z}^{p}_{-1}(r)+ \mathrm{Z}^{p}_{0}(r)\right) +P_{R}\left( 2\mathrm{Z}^{p}_{N+1}(r)+\mathrm{Z}^{p}_{N+1}(r)\right)\\ &\,\,+\sum_{i=1}^{N-1}\alpha_{i}^{(1)}(t)\mathrm{Z}^{p}_{i}(r)-\alpha_1^{(1)}(t)\mathrm{Z}^{p}_{-1}(r)-\alpha_{N-1}^{(1)}(t)\mathrm{Z}^{p}_{N+1}(r),\\
   	&=\mathrm{Q}_{p}(r)+ \mathds{Z}_{p}^{T}(r)\alpha^{(1)}(t).
   \end{align*}
   Similarly we have
   \begin{align*}
   	\sum_{i=-1}^{N+1}\alpha_{i}^{(1)}(t)\overline{\mathrm{Z}}^{p}_{i}(r)&=\overline{\mathrm{Q}_{p}}(r)+ \overline{\mathds{Z}}_{p}^{T}(r)\alpha^{(1)}(t),\\
   		\sum_{i=-1}^{N+1}\alpha_{i}^{(2)}(t)\mathrm{Z}^{s}_{i}(r)&=\mathrm{Q}_{s}(r)+ \mathds{Z}_{s}^{T}(r)\alpha^{(2)}(t),\\
   			\sum_{i=-1}^{N+1}\alpha_{i}^{(2)}(t)\overline{\mathrm{Z}}^{s}_{i}(r)&=\overline{\mathrm{Q}_{s}}(r)+ \overline{\mathds{Z}}_{s}^{T}(r)\alpha^{(2)}(t).
   \end{align*}
   Now we express $\widetilde{U_{h}}$ at the points $r_{i}$ for $i = 0, \dots,N$. We start by the equation (\ref{Pa}), at $r=r_{0}$, we obtain   
\begin{align}
	\dfrac{d\alpha_{0}^{(1)}}{dt}(t)=&
	-\left( \mathrm{Q}_{p}(r_{0})+ \mathds{Z}_{p}^{T}(r_{0})\alpha^{(1)}(t)\right) \left[   \dfrac{1}{r_{0}^{d-1}}\left( \dfrac{\partial ( r^{d-1}\widetilde{\alpha_{p}})}{\partial r}(t,r_{0})\right)P_{L}(t) 
	+\widetilde{\alpha_{p}}(t,r_{0})\frac{1}{h}(\alpha_{1}^{(1)}(t)-P_{L}(t))\right]\nonumber \\
	&-\widetilde{\alpha_{p}}(t,r_{0})\left( \overline{\mathrm{Q}_{p}}(r_{0})+ \overline{\mathds{Z}}_{p}^{T}(r_{0})\alpha^{(1)}(t)\right) P_{L}(t)
	+\left[H\left( M_{L}(t)\right) -a_{1}\lambda(t)M_{L}(t)\right]P_{L}(t),\nonumber\\
\dfrac{d\alpha_{0}^{(2)}}{dt}(t)=&
	-\left( \mathrm{Q}_{s}(r_{0})+ \mathds{Z}_{s}^{T}(r_{0})\alpha^{(2)}(t)\right) \left[   \dfrac{1}{r_{0}^{d-1}}\left( \dfrac{\partial ( r^{d-1}\widetilde{\alpha_{s}})}{\partial r}(t,r_{0})\right)S_{L}(t) 
	+\widetilde{\alpha_{s}}(t,r_{0})\frac{1}{h}(\alpha_{1}^{(1)}(t)-S_{L}(t))\right] \nonumber\\
	&-\widetilde{\alpha_{s}}(t,r_{0})\left( \overline{\mathrm{Q}_{s}}(r_{0})+ \overline{\mathds{Z}}_{s}^{T}(r_{0})\alpha^{(2)}(t)\right) S_{L}(t)
 -a_{2}\lambda(t)M_{L}(t)S_{L}(t),\nonumber\\
\dfrac{d\alpha_{0}^{(3)}}{dt}(t)=& \frac{D(d-1)}{r h}\left(\alpha_{1}^{(3)}(t)-M_{L}(t)\right)+
 M_{s}S_{L}(t)\left( 1-M_{L}(t))\right) -\eta M_{L}(t)P_{L}(t).\label{NL1}
 \end{align}
 Now, substituting at $r=r_{i}$, for $i=1,\dots,N-1$, we get the following
\begin{multline}
	\label{NL2}
	\frac{1}{6}\left(\dfrac{d\alpha_{i-1}^{(1)}}{dt}(t)+4\dfrac{d\alpha_{i}^{(1)}}{dt}(t)+
	\dfrac{d\alpha_{i+1}^{(1)}}{dt}(t) \right) =\\
	\resizebox{0.95\textwidth}{!}{$-\left( \mathrm{Q}_{p}(r_{i})+ \mathds{Z}_{p}^{T}(r_{i})\alpha^{(1)}(t)\right) \left[    \dfrac{1}{r_{i}^{d-1}}\left( \dfrac{\partial ( r^{d-1}\widetilde{\alpha_{p}})}{\partial r}(t,r_{i})\right)	\frac{1}{6}\left(\alpha_{i-1}^{(1)}(t)+4\alpha_{i}^{(1)}(t)+
	\alpha_{i+1}^{(1)}(t) \right) 
	+\widetilde{\alpha_{p}}(t,r_{i})\frac{1}{2h}\left( \alpha_{i+1}^{(1)}(t)-\alpha_{i-1}^{(1)}(t)\right)\right] $} \\
	-\frac{1}{6}\widetilde{\alpha_{p}}(t,r_{i})\left( \overline{\mathrm{Q}_{p}}(r_{i})+ \overline{\mathds{Z}}_{p}^{T}(r_{i})\alpha^{(1)}(t)\right) \left(\alpha_{i-1}^{(1)}(t)+4\alpha_{i}^{(1)}(t)+
	\alpha_{i+1}^{(1)}(t) \right)\\ 
	\resizebox{0.95\textwidth}{!}{$+\left[H\left(\frac{1}{6}\left(\alpha_{i-1}^{(1)}(t)+4\alpha_{i}^{(1)}(t)+
	\alpha_{i+1}^{(1)}(t) \right)\right) 
	  -\frac{1}{6}a_{1}\lambda(t)\left(\alpha_{i-1}^{(3)}(t)+4\alpha_{i}^{(3)}(t)+
	\alpha_{i+1}^{(3)}(t) \right)\right] \frac{1}{6}\left(\alpha_{i-1}^{(1)}(t)+4\alpha_{i}^{(1)}(t)+
	\alpha_{i+1}^{(1)}(t) \right) .$}
\end{multline}
\begin{multline}
	\label{NL3}
	\frac{1}{6}\left(\dfrac{d\alpha_{i-1}^{(2)}}{dt}(t)+4\dfrac{d\alpha_{i}^{(2)}}{dt}(t)+
	\dfrac{d\alpha_{i+1}^{(2)}}{dt}(t) \right) =\\
	\resizebox{0.95\textwidth}{!}{$-\left( \mathrm{Q}_{s}(r_{i})+ \mathds{Z}_{s}^{T}(r_{i})\alpha^{(2)}(t)\right) \left[    \dfrac{1}{r_{i}^{d-1}}\left( \dfrac{\partial ( r^{d-1}\widetilde{\alpha_{s}})}{\partial r}(t,r_{i})\right)	\frac{1}{6}\left(\alpha_{i-1}^{(2)}(t)+4\alpha_{i}^{(2)}(t)+
	\alpha_{i+1}^{(2)}(t) \right) 
	+\widetilde{\alpha_{s}}(t,r_{i})\frac{1}{2h}\left( \alpha_{i+1}^{(2)}(t)-\alpha_{i-1}^{(2)}(t)\right)\right] $} \\
	-\frac{1}{6}\widetilde{\alpha_{s}}(t,r_{i})\left( \overline{\mathrm{Q}_{s}}(r_{i})+ \overline{\mathds{Z}}_{s}^{T}(r_{i})\alpha^{(2)}(t)\right) \left(\alpha_{i-1}^{(2)}(t)+4\alpha_{i}^{(2)}(t)+
	\alpha_{i+1}^{(2)}(t) \right)\\ 
	 -\frac{1}{36}a_{2}\lambda(t)\left(\alpha_{i-1}^{(3)}(t)+4\alpha_{i}^{(3)}(t)+
	\alpha_{i+1}^{(3)}(t) \right) \left(\alpha_{i-1}^{(2)}(t)+4\alpha_{i}^{(2)}(t)+
	\alpha_{i+1}^{(2)}(t) \right) .
\end{multline}
\begin{multline}
	\label{NL4}
	\frac{1}{6}\left(\dfrac{d\alpha_{i-1}^{(3)}}{dt}(t)+4\dfrac{d\alpha_{i}^{(3)}}{dt}(t)+
	\dfrac{d\alpha_{i+1}^{(3)}}{dt}(t) \right)=\\
	D\left( \frac{d-1}{2rh}\left( \alpha_{i+1}^{(3)}(t)-\alpha_{i-1}^{(3)}(t)\right)+\frac{1}{h^{2}}\left( \alpha_{i-1}^{(3)}(t)-2\alpha_{i}^{(3)}(t)+
	\alpha_{i+1}^{(3)}(t)\right) \right)\\
	+\frac{M_{s}}{6} \left(\alpha_{i-1}^{(2)}(t)+4\alpha_{i}^{(2)}(t)+
	\alpha_{i+1}^{(2)}(t) \right)\left( 1-\frac{1}{6}\left( \alpha_{i-1}^{(3)}(t)+4\alpha_{i}^{(3)}(t)+
	\alpha_{i+1}^{(3)}(t) \right) \right)\\
	-\eta\dfrac{1}{36} \left(\alpha_{i-1}^{(3)}(t)+4\alpha_{i}^{(3)}(t)+
	\alpha_{i+1}^{(3)}(t) \right) \left(\alpha_{i-1}^{(1)}(t)+4\alpha_{i}^{(1)}(t)+
	\alpha_{i+1}^{(1)}(t) \right)   .
\end{multline}
And finally we get for $r=r_{N}=R$ the following
\begin{align}
	\dfrac{d\alpha_{N}^{(1)}}{dt}(t)=&
	-\left( \mathrm{Q}_{p}(r_{N})+ \mathds{Z}_{p}^{T}(r_{N})\alpha^{(1)}(t)\right) \left[    \dfrac{1}{r_{N}^{d-1}}\left( \dfrac{\partial ( r^{d-1}\widetilde{\alpha_{p}})}{\partial r}(t,r_{N})\right)P_{R}
	+\widetilde{\alpha_{p}}(t,r_{N})\frac{1}{h}\left( P_{R}-\alpha_{N-1}^{(1)}(t)\right)\right] \nonumber\\
	&-\widetilde{\alpha_{p}}(t,r_{N})\left( \overline{\mathrm{Q}_{p}}(r_{N})+ \overline{\mathds{Z}}_{p}^{T}(r_{N})\alpha^{(1)}(t)\right) P_{R}
	+\left[H\left( M_{R}\right) -a_{1}\lambda(t)M_{R}\right]P_{R},\nonumber\\
	\dfrac{d\alpha_{N}^{(2)}}{dt}(t)=&
	-\left( \mathrm{Q}_{s}(r_{N})+ \mathds{Z}_{s}^{T}(r_{N})\alpha^{(2)}(t)\right) \left[    \dfrac{1}{r_{N}^{d-1}}\left( \dfrac{\partial ( r^{d-1}\widetilde{\alpha_{s}})}{\partial r}(t,r_{N})\right)S_{R} 
	+\widetilde{\alpha_{s}}(t,r_{N})\frac{1}{h}\left( P_{R}-\alpha_{N-1}^{(1)}(t)\right)\right] \nonumber\\
	&-\widetilde{\alpha_{s}}(t,r_{N})\left( \overline{\mathrm{Q}_{s}}(r_{N})+ \overline{\mathds{Z}}_{s}^{T}(r_{N})\alpha^{(2)}(t)\right) S_{R}
	-a_{2}\lambda(t)M_{R}S_{R},\nonumber\\
	\dfrac{d\alpha_{N}^{(3)}}{dt}(t)=& \frac{D(d-1)}{rh}\left( M_R-\alpha_{N-1}^{(3)}(t)\right)+
 M_{s}S_{R}(t)\left( 1-M_{R}\right) 
	 -\eta M_{R}P_{R}.\label{NL5}
\end{align}
For the initial time $t=0$, we have 
\begin{equation}
	\frac{1}{6}\left(\begin{array}{c}
		4\alpha_{1}^{(1)}(0)+\alpha_{2}^{(1)}(0) \\
		\alpha_{1}^{(1)}(0)+4\alpha_{2}^{(1)}(0)+\alpha_{3}^{(1)}(0) \\
		\vdots \\
			
		\alpha_{N-3}^{(1)}(0)+4\alpha_{N-2}^{(1)}(0)+\alpha_{N-1}^{(1)}(0) \\
		\alpha_{N-2}^{(1)}(0)+4\alpha_{N-1}^{(1)}(0) \\
			4\alpha_{1}^{(2)}(0)+\alpha_{2}^{(2)}(0) \\
		\alpha_{1}^{(2)}(0)+4\alpha_{2}^{(2)}(0)+\alpha_{3}^{(2)}(0) \\
		\vdots \\
		
		\alpha_{N-3}^{(2)}(0)+4\alpha_{N-2}^{(2)}(0)+\alpha_{N-1}^{(2)}(0) \\
		\alpha_{N-2}^{(2)}(0)+4\alpha_{N-1}^{(2)}(0) \\
			4\alpha_{1}^{(3)}(0)+\alpha_{2}^{(3)}(0) \\
		\alpha_{1}^{(3)}(0)+4\alpha_{2}^{(3)}(0)+\alpha_{3}^{(3)}(0) \\
		\vdots \\
		
		\alpha_{N-3}^{(3)}(0)+4\alpha_{N-2}^{(3)}(0)+\alpha_{N-1}^{(3)}(0) \\
		\alpha_{N-2}^{(3)}(0)+4\alpha_{N-1}^{(3)}(0) \\
	\end{array}\right)=\left(\begin{array}{c}
	\widetilde{P}_{0}(r_{1})-\frac{1}{6}\widetilde{P}_{0}(r_{0})\\
	\widetilde{P}_{0}(r_{2}) \\
	\vdots \\
	
	\widetilde{P}_{0}(r_{N-2}) \\
	\widetilde{P}_{0}(r_{N-1})-\frac{1}{6}\widetilde{P}_{0}(r_{N}) \\
	\widetilde{S}_{0}(r_{1})-\frac{1}{6}\widetilde{S}_{0}(r_{0})\\
	\widetilde{S}_{0}(r_{2}) \\
	\vdots \\
	
	\widetilde{S}_{0}(r_{N-2}) \\
	\widetilde{S}_{0}(r_{N-1})-\frac{1}{6}\widetilde{S}_{0}(r_{N}) \\
\widetilde{M}_{0}(r_{1})-\frac{1}{6}\widetilde{M}_{0}(r_{0})\\
\widetilde{M}_{0}(r_{2}) \\
\vdots \\

\widetilde{M}_{0}(r_{N-2}) \\
\widetilde{M}_{0}(r_{N-1})-\frac{1}{6}\widetilde{M}_{0}(r_{N}) \\
	\end{array}\right):=\widetilde{U}^{0}.
\end{equation}
Which can be expressed as $\mathds{A}\upalpha(0)=\widetilde{U}^{0}$, where $\mathds{A}$ is a $3\left( N-1\right)\times3\left( N-1\right)$ block matrix defined as 
\begin{equation}
\label{MatA}
	\mathds{A}=
	\begin{bmatrix}
		A & 0 & 0\\
		0 & A & 0 \\
		0 & 0 & A
	\end{bmatrix}.
\end{equation}
Where A is the following $\left(  N-1\right)\times\left( N-1\right)$ matrix 
\begin{equation}
	A=\frac{1}{6}\left(\begin{array}{ccccc}
		4 & 1 & 0 & \cdots & 0 \\
		1 & 4 & 1 & & \vdots \\
		0 & \ddots & \ddots & \ddots & 0 \\
		\vdots & & 1 & 4 & 1 \\
		0 & \cdots & 0 & 1 & 4
	\end{array}\right) .
\end{equation}
Furthermore, let define the matrices of size $(N-1)\times (N-1)$ as follows:
\begin{align}
\label{matphi}
 A^{\prime}:=\frac{-1}{2 h}\left(\begin{array}{ccccc}\nonumber
	0 & -1 & 0 & \cdots & 0 \\
		1 & 0 & -1 & & \vdots \\
		0 & \ddots & \ddots & \ddots & 0 \\
		\vdots & & 1 & 0 & -1 \\
		0 & \cdots & 0 & 1 & 0
	\end{array}\right), \quad&\quad
 A^{\prime \prime}:=\frac{1}{h^2}\left(\begin{array}{ccccc}
		-2 & 1 & 0 & \cdots & 0 \\
		1 & -2 & 1 & & \vdots \\
		0 & \ddots & \ddots & \ddots & 0 \\
		\vdots & & 1 & -2 & 1 \\
		0 & \cdots & 0 & 1 & -2
	\end{array}\right),\\
    \mathbb{M}_{Z}^k:=\Bigl[\mathds{Z}_{k}(r_{1}),\dots,
    \mathds{Z}_{k}(r_{N-1})\Bigr]^T,\quad &\text{and}\quad \overline{\mathbb{M}}_{Z}^k:=\Bigl[\overline{\mathds{Z}}_{k}(r_{1}),\dots,
    \overline{\mathds{Z}}_{k}(r_{N-1})\Bigr]^T, \qquad\quad k=p,s.
\end{align}
Let also define for $k=p,s,$ the following vectors of size $N-1$ as follows:
\begin{align}
\label{vectphi}
      \mathbb{M}_{Q}^k:=\Bigl(Q_k(r_1),\dots,Q_k(r_{N-1})\Bigr)^T,\quad &\quad\overline{\mathbb{M}}_{Q}^k:=\Bigl(\overline{Q}_{k}(r_{1}),\dots,
    \overline{Q}_{k}(r_{N-1})\Bigr)^T,\nonumber\\
    \mathbb{M}_{\tilde{\alpha}}^k(t):=\Bigl(\widetilde{\alpha_{k}}(t,r_{1}), \dots , \widetilde{\alpha_{k}}(t,r_{N-1})\Bigr)^T,\quad&\quad\overline{\mathbb{M}}_{\tilde{\alpha}}^k(t):=\Bigl(\dfrac{1}{r_1^{d-1}}\dfrac{\partial}{\partial r}\Bigl(r^{d-1}\widetilde{\alpha_{p}}(t,r_1) \Bigr),\dots,
    \dfrac{1}{r_{N-1}^{d-1}}\dfrac{\partial}{\partial r}\Bigl(r^{d-1}\widetilde{\alpha_{p}}(t,r_{N-1}) \Bigr)\Bigr)^T,\nonumber\\
    V_M(t)=\Bigl(\frac{M_L(t)}{6},0,\dots,0,\frac{M_R}{6}\Bigr)^T,\quad&\qquad\overline{r}_d:=(d-1)\Bigl(\dfrac{1}{r_1^{d-1}},\dots,
    \dfrac{1}{r_{N-1}^{d-1}}\Bigr)^T, \quad
    \text{and}\quad \mathds{1}=(1,\dots,1)^T.
\end{align}
		Subsequently, $\upalpha$ satisfies the following ordinary differential equation (ODE):
	\begin{equation}
		\label{ODEBSPLINE}
		\left\lbrace \begin{array}{ll}
			\displaystyle
			\mathds{A}\dfrac{d\upalpha}{dt}(t)&=\Upphi(t,\upalpha(t)),\qquad t\in[0,T],\\
			\mathds{A}\upalpha(0)&=\widetilde{U}^{0}.
		\end{array}\right.
	\end{equation}
 Where $\Upphi$ is a nonlinear function defined by the equation (\ref{NL1}),(\ref{NL2}),(\ref{NL3}),(\ref{NL4}) and (\ref{NL5}). Then, by using the notations \eqref{MatA},\eqref{matphi} and \eqref{vectphi}, $\Upphi$ may be written as follows:
\begin{eqnarray*}
   \Upphi(t,\upalpha(t))=
	-\left(\begin{bmatrix}
		 \mathbb{M}_{Z}^p& 0 & 0\\
		0 & \mathbb{M}_{Z}^s & 0 \\
		0 & 0 & 0
	\end{bmatrix}\upalpha(t)+\begin{bmatrix}
		 \mathbb{M}_{Q}^p\\
		\mathbb{M}_{Q}^s \\
		0 
	\end{bmatrix}\right)\odot \left(\begin{bmatrix}
		 \overline{\mathbb{M}}_{\tilde{\alpha}}^p(t)\\
		\overline{\mathbb{M}}_{\tilde{\alpha}}^s(t) \\
		0 
	\end{bmatrix} \odot\mathds{A}+\begin{bmatrix}
		 \mathbb{M}_{\tilde{\alpha}}^p\\
		\mathbb{M}_{\tilde{\alpha}}^s \\
		0 
	\end{bmatrix}\odot\begin{bmatrix}
		A^{\prime} & 0 & 0\\
		0 & A^{\prime} & 0 \\
		0 & 0 & 0
	\end{bmatrix} \right)\upalpha(t)\\
 +
 \begin{bmatrix}
		 \left(H\left(A\alpha^{(3)}(t)+V_M(t)\right)-a_{1}\lambda(t)A\alpha^{(3)}(t)\right)\odot\left( A \alpha^{(1)}(t)\right)\\
		-a_{2}\lambda(t)\left(A\alpha^{(3)}(t)\right)\odot\left( A \alpha^{(1)}(t)\right) \\
		D\left(\overline{r}_d\odot A^{\prime}+A^{\prime \prime}\right)\alpha^{(3)}(t)+M_s\left( A \alpha^{(2)}(t)\right)\odot\left(\mathds{1}-A \alpha^{(3)}(t)\right)-\eta\left( A \alpha^{(3)}(t)\right)\odot\left( A \alpha^{(1)}(t) \right)
	\end{bmatrix}+f(t,\upalpha(t)).
\end{eqnarray*}
Where $\odot$ denotes for the Hadamard product, and $f$ is a $3\times(N-1)$ vector, expressed as follows:
\begin{eqnarray*}
f_1(t,\upalpha(t))=\frac{-1}{6}\dfrac{d\alpha_{0}^{(1)}}{dt}(t)-\left( \mathrm{Q}_{p}(r_{1})+ \mathds{Z}_{p}^{T}(r_{1})\alpha^{(1)}(t)\right) \left[\dfrac{1}{6 r_{1}^{d-1}}\left( \dfrac{\partial ( r^{d-1}\widetilde{\alpha_{p}})}{\partial r}(t,r_{1})\right)P_{L}(t) 
	-\widetilde{\alpha_{p}}(t,r_{1})\frac{1}{2 h} P_{L}(t)\right]\\
 -\frac{1}{6}\widetilde{\alpha_{p}}(t,r_{1})\left( \overline{\mathrm{Q}_{p}}(r_{1})+ \overline{\mathds{Z}}_{p}^{T}(r_{1})\alpha^{(1)}(t)\right)P_{L}(t)\hspace{3cm}\\
 -\frac{1}{36}a_1 \lambda(t)\Bigl[M_L(t)\left(P_L(t)+4\alpha_1^{(1)}(t)+\alpha_2^{(1)}(t)\right)+P_L(t)\left(M_L(t)+4\alpha_1^{(3)}(t)+\alpha_2^{(3)}(t)\right)\Bigr],
\end{eqnarray*}
\begin{eqnarray*}
f_{N-1}(t,\upalpha(t))=\frac{-1}{6}\dfrac{d\alpha_{N}^{(1)}}{dt}(t)-\left( \mathrm{Q}_{p}(r_{N-1})+ \mathds{Z}_{p}^{T}(r_{N-1})\alpha^{(1)}(t)\right) \left[\dfrac{1}{6 r_{N-1}^{d-1}}\left( \dfrac{\partial ( r^{d-1}\widetilde{\alpha_{p}})}{\partial r}(t,r_{N-1})\right)P_{R}
	+\widetilde{\alpha_{p}}(t,r_{N-1})\frac{1}{2 h} P_{R}\right]\\
 -\frac{1}{6}\widetilde{\alpha_{p}}(t,r_{N-1})\left( \overline{\mathrm{Q}_{p}}(r_{N-1})+ \overline{\mathds{Z}}_{p}^{T}(r_{N-1})\alpha^{(1)}(t)\right)P_{R}\hspace{3cm}\\
 -\frac{1}{36}a_1 \lambda(t)\Bigl[M_R\left(P_R+4\alpha_{N-1}^{(1)}(t)+\alpha_{N-2}^{(1)}(t)\right)+P_R\left(M_R+4\alpha_{N-1}^{(3)}(t)+\alpha_{N-2}^{(3)}(t)\right)\Bigr],
\end{eqnarray*}
\begin{eqnarray*}
f_N(t,\upalpha(t))=\frac{-1}{6}\dfrac{d\alpha_{0}^{(2)}}{dt}(t)-\left( \mathrm{Q}_{s}(r_{1})+ \mathds{Z}_{s}^{T}(r_{1})\alpha^{(2)}(t)\right) \left[\dfrac{1}{6 r_{1}^{d-1}}\left( \dfrac{\partial ( r^{d-1}\widetilde{\alpha_{s}})}{\partial r}(t,r_{1})\right)S_{L}(t) 
	-\widetilde{\alpha_{s}}(t,r_{1})\frac{1}{2 h} S_{L}(t)\right]\\
 -\frac{1}{6}\widetilde{\alpha_{s}}(t,r_{1})\left( \overline{\mathrm{Q}_{s}}(r_{1})+ \overline{\mathds{Z}}_{s}^{T}(r_{1})\alpha^{(2)}(t)\right)S_{L}(t)\hspace{3cm}\\
 -\frac{1}{36}a_2 \lambda(t)\Bigl[M_L(t)\left(S_L(t)+4\alpha_1^{(2)}(t)+\alpha_2^{(2)}(t)\right)+S_L(t)\left(M_L(t)+4\alpha_1^{(3)}(t)+\alpha_2^{(3)}(t)\right)\Bigr],
\end{eqnarray*}
\begin{eqnarray*}
   f_{2N-2}(t,\upalpha(t))=\frac{-1}{6}\dfrac{d\alpha_{N}^{(2)}}{dt}(t)-\left( \mathrm{Q}_{S}(r_{N-1})+ \mathds{Z}_{S}^{T}(r_{N-1})\alpha^{(2)}(t)\right) \left[\dfrac{1}{6 r_{N-1}^{d-1}}\left( \dfrac{\partial ( r^{d-1}\widetilde{\alpha_{s}})}{\partial r}(t,r_{N-1})\right)S_{R}
	+\widetilde{\alpha_{s}}(t,r_{N-1})\frac{1}{2 h} S_{R}\right]\\
 -\frac{1}{6}\widetilde{\alpha_{s}}(t,r_{N-1})\left( \overline{\mathrm{Q}_{s}}(r_{N-1})+ \overline{\mathds{Z}}_{s}^{T}(r_{N-1})\alpha^{(2)}(t)\right)S_{R}\hspace{3cm}\\
 -\frac{1}{36}a_2 \lambda(t)\Bigl[M_R\left(S_R+4\alpha_{N-1}^{(2)}(t)+\alpha_{N-2}^{(2)}(t)\right)+S_R\left(M_R+4\alpha_{N-1}^{(3)}(t)+\alpha_{N-2}^{(3)}(t)\right)\Bigr],
\end{eqnarray*}
\begin{eqnarray*}
f_{2N-1}(t,\upalpha(t))=\frac{-1}{6}\dfrac{d\alpha_{0}^{(3)}}{dt}(t)+D\left( \frac{1}{h^2}M_L(t)-\frac{d-1}{2 r h}M_L(t)\right)+\frac{M_s}{6} S_L(t)
 -\frac{M_s}{36}\Bigl[M_L(t)\left(S_L(t)+4\alpha_1^{(2)}(t)+\alpha_2^{(2)}(t)\right)\ \hspace{3cm}\\ +S_L(t)\left(M_L(t)+4\alpha_1^{(3)}(t)+\alpha_2^{(3)}(t)\right)\Bigr]
 -\frac{\eta}{36}\Bigl[M_L\left(P_L(t)+4\alpha_1^{(1)}(t)+\alpha_2^{(1)}(t)\right)+P_L(t)\left(M_L(t)+4\alpha_1^{(3)}(t)+\alpha_2^{(3)}(t)\right)\Bigr],\hspace{2cm}
\end{eqnarray*}
\begin{eqnarray*}
   f_{3N-3}(t,\upalpha(t))=\frac{-1}{6}\dfrac{d\alpha_{0}^{(3)}}{dt}(t)+D\left( \frac{1}{h^2}M_R+\frac{d-1}{2 r h}M_R\right)+\frac{M_s}{6} S_R
 -\frac{M_s}{36}\Bigl[M_R\left(S_R+4\alpha_{N-1}^{(2)}(t)+\alpha_{N-2}^{(2)}(t)\right)\ \hspace{3.3cm}\\ +S_R\left(M_L(t)+4\alpha_{N-1}^{(3)}(t)+\alpha_{N-2}^{(3)}(t)\right)\Bigr]
 -\frac{\eta}{36}\Bigl[M_R\left(P_L(t)+4\alpha_{N-1}^{(1)}(t)+\alpha_{N-2}^{(1)}(t)\right)+P_R\left(M_R+4\alpha_{N-1}^{(3)}(t)+\alpha_{N-2}^{(3)}(t)\right)\Bigr],\hspace{1cm}
\end{eqnarray*}
and $f_j(t,\upalpha(t))=0$ for all $j\notin\Bigl\{1,N-1,N,2N-2,2N-1,3N-3\Bigr\}.$

The solution of the ODE (\ref{ODEBSPLINE}) can be approached using several methods found in the literature. However, due to the stiffness of the equation, which is a common characteristic in systems derived from discretized partial differential equations.

Common explicit methods, such as the popular Runge-Kutta family of integrators, are not suitable for such stiff problems. To overcome this limitation, we employ the Backward Differentiation Formula (BDF) method, which is an implicit multi-step method specifically designed for solving stiff ordinary differential equations. This method allows for larger time steps while maintaining accuracy and stability, which can be done by using information from previous time steps and solving an implicit equation at each step.
Here, we apply the BDF method of order $p=6$, for maximal stability and efficiency. The method applied to (\ref{ODEBSPLINE}) gives the following approximations
\begin{equation}
	\label{bdfequation}
	\mathds{A} \upalpha_n-\beta \Delta t\Upphi\left(t_{n},\upalpha_n\right)-\mathds{A}\sum_{j=0}^{p-1} \eta_j  \upalpha_{n-j-1}=0.
\end{equation}
In this context, \( \Delta t =\dfrac{T}{p}\) represents the time step, and \( \upalpha_n = [\upalpha_{1}^{n}, \ldots, \upalpha_{N-1}^{n}]^T \) is the BDF method's approximation of the vector \( \upalpha(t) \) at time \( t_n = n \Delta t \) for $n = 0,\ldots,p$. The coefficients \( \eta_j \) and \( \beta \) for the \( p \)-step BDF formula are provided in Table \ref{table:bdf}.
\begin{table}[H]
	\centering
	\begin{tabular}{|c||c|c|c|c|c|c|c|}
		\hline $p$ & $\beta$ & $\eta_{0}$ & $\eta_1$ & $\eta_{2}$ & $\eta_{3}$&$\eta_{4}$&$\eta_{5}$\\
\hline \hline $1$ & $1$ & $1$ & & & &&\\
		\hline $2$ & $\frac{2}{3}$ & $\frac{4}{3}$ & $\frac{-1}{3}$ &  & &&\\
	    \hline $3$ & $\frac{6}{11}$ & $\frac{18}{11}$ & $\frac{-9}{11}$ & $\frac{2}{11}$ & &&\\
		\hline $4$ & $\frac{12}{25}$ & $\frac{48}{25}$ & $\frac{-36}{25}$ & $\frac{16}{25}$ &$\frac{-3}{25}$ &&\\ 
		\hline $5$ & $\frac{60}{137}$  & $\frac{300}{137}$ & $\frac{-300}{137}$ & $\frac{200}{137}$ &$\frac{-75}{137}$ &$\frac{12}{137}$&\\ 
			\hline $6$ & $\frac{60}{147}$ & $\frac{360}{147}$ & $\frac{-450}{147}$ & $\frac{400}{147}$ &$\frac{-225}{147}$ &$\frac{72}{147}$&$\frac{-10}{147}$\\ 
		\hline
	\end{tabular}
	\caption{BDF coefficients for order 6.}
	\label{table:bdf}
\end{table}

At each time step $n$, the equation (\ref{bdfequation}) for $\upalpha_n$ must be solved by reformulating it into the following form:
\begin{equation}
    \upvarphi\left(\upalpha_n\right)=\mathds{A}\upalpha_n-\beta \Delta t\Upphi(t_n,\upalpha_n)-\mathds{A}\sum_{j=0}^{p-1} \eta_j  \upalpha_{n-j-1}=0.
\end{equation}
To solve the equation (\ref{bdfequation}) for $\upalpha_n$ efficiently, we employ the Newton method. We used the solution from the previous time step as an initial guess. The Newton method approximates $\upalpha_n$ through a series of iterations $\left(\zeta_k\right)$ as follows:
\begin{equation}
\left\{\begin{array}{l}
	\zeta_0,\\
	\zeta_{k+1}=\zeta_k-\left[J_{\upvarphi}\left(\zeta_k\right)\right]^{-1} \upvarphi \left(\zeta_k\right), \quad k \geqslant 0,
\end{array}\right.
\end{equation}
where $J_{\upvarphi}\left(\zeta_k\right)$ is the Jacobian matrix of $\upvarphi$, expressed as follows:
\begin{equation}
J_{\upvarphi}\left(\zeta_k\right)=\mathds{A-}\beta \Delta t J_{\Upphi}\left(\zeta_k\right),
\end{equation}
with $J_{\Upphi}$ denotes for the Jacobian matrix of $\Upphi$ with respect to $\upalpha_n$.
\section{Numerical experiments}
	\label{sec5}
In this section, we aim to examine the accuracy of our proposed numerical scheme by comparing it with a known analytical solution. This evaluation is essential to determine the efficiency and validity of the method. The construction of a classical solution for a simplified version of the system presented in equation (\ref{radialmodel}) presents significant challenges, mainly due to the nonlocal term, which must be calculated analytically in order to prevent errors that could arise from numerical integration. Then in section \ref{sub:6.2}, we use values similar to those reported by Lefebvre in \cite{lefebvre2017spatial}, examining the dynamics of avascular tumor growth in the absence of medical treatment, considering a nonlocal velocity.

Throughout this section, we consider the spatial dimension of our system \eqref{equationU} to be $d=2$. In this context, the nonlocal terms are expressed as follows:
 \begin{align*}
 	\mathds{Z}^{k}_{i}(r)&:= \left(\tilde{\gamma_{p}}\tilde{*}   B_i\right)(r)=	2\int^{R}_{0}\int^{\pi}_{0}\tilde{\gamma_{p}}\left( \sqrt{(s^{2}+
 		r^{2}+2rs\cos(\theta))}\right) B(s)s \,ds\, d \theta  \quad\text{and}\\
 		\overline{\mathrm{Z}}^{k}_{i}(r)&:= \dfrac{\partial \left(\tilde{\gamma_{p}}\tilde{*}  B_i\right)(r)}{\partial r}=	2\int^{R}_{0}\int^{\pi}_{0}\dfrac{r + s\cos(\theta)}{\sqrt{r^2 + s^2 + 2rs\cos(\theta)}}\tilde{\gamma_{p}}^{\prime}\left( \sqrt{(s^{2}+r^{2}+2rs\cos(\theta))}\right) B(s)s \,ds\, d \theta.
 \end{align*}
 These double integrals are calculated numerically using the function \texttt{integral2} in MATLAB, which is effective for this type of two-dimensional integration over rectangular regions.
 
The simulations performed in section \ref{sub:6.1} are executed in the spatial domain $[0,1]$, while in section \ref{sub:6.2}, the spatial domain $[10^{-3}, 2]$ is considered. It is important to note that radial hyperbolic and radial parabolic equations are likely to suffer from singularities at the origin (see \cite{mohseni2000numerical}). To mitigate this issue, the calculations are initiated from a small radius in space. 
The numerical simulations presented in this paper were performed using MATLAB on a standard laptop computer equipped with an Intel Core i5 processor (2.6 GHz, 2 cores) and 8GB of RAM. No specialized GPU acceleration was utilized for these computations. This configuration was sufficient for running our simulations, demonstrating the efficiency of the proposed numerical method. 
\subsection{Convergence test with known analytical solution}\label{sub:6.1}
In this section, we consider the following coupled system with source terms 
\begin{equation}
\resizebox{0.95\textwidth}{!}{$
	\label{radialmodeltest}
	\left\lbrace \begin{array}{ll}
		\vspace{0.2cm}
		\dfrac{\partial \widetilde{P}}{\partial t}(t,r)+2\pi r\dfrac{\partial }{\partial r}\left( \left( \int^{R}_{0}\widetilde{P}(t,s) s ds \right) \widetilde{P}\right) (t,r)=\left(H(\widetilde{M})(t,r)-t\widetilde{M}(t,r)-2\left( \int^{R}_{0}\widetilde{P}(t,s) s ds \right)\right)\widetilde{P}(t,r)+f_{1}(t,r), \\
		\vspace{0.2cm}
		\dfrac{\partial \widetilde{S}}{\partial t}(t,r)+2\pi r\dfrac{\partial }{\partial r}\left( \left( \int^{R}_{0}\widetilde{S}(t,s) s ds \right) \widetilde{S}\right) (t,r) =-\left( t\widetilde{M}(t,r)+2\left( \int^{R}_{0}\widetilde{S}(t,s) s ds \right) \right)\widetilde{S}(t,r)+f_{2}(t,r),&(t,r)\in [0,T]\times [0,R]\\
		\vspace{0.2cm}
		\dfrac{\partial \widetilde{M}}{\partial t}(t,r)-0.01\left(r \dfrac{\partial^{2} \widetilde{M}}{\partial r^{2}}(t,r)+\dfrac{\partial\widetilde{M}}{\partial r}(t,r)\right) =\widetilde{S}(t,x)(1-\widetilde{M}(t,x))- \widetilde{M}(t,x) \widetilde{P}(t,x)+f_{3}(t,r), &(t,r)\in [0,T]\times [0,R],\\
		\vspace{0.2cm}
		\widetilde{U}(t,R)=\widetilde{U}_{ex}(t,R),\quad \widetilde{U}(t,0)=\widetilde{U}_{ex}(t,0), &t\in [0,T], \\
		\vspace{0.2cm}
		\widetilde{U}(0,r)=\widetilde{U}_{ex}(0,r), &r \in [0,R].
	\end{array}\right.$}
\end{equation}

The source term $\left(f(t,r)=f_{1}(t,r),f_{2}(t,r),f_{3}(t,r)\right)$ is chosen in order that the system (\ref{radialmodeltest}) has the following analytical solution
\begin{equation}
	\label{Uex}
	\widetilde{U}_{ex}(t,r)=\left( \begin{array}{ll}
		\widetilde{P}_{ex}(t,r)\\
		\widetilde{S}_{ex}(t,r)\\
		\widetilde{M}_{ex}(t,r)
	\end{array}\right) = \left(\begin{array}{ll} 
		\dfrac{3}{2}\exp\left( \dfrac{\left(r-\frac{t}{2000}\right)^{2}}{50}\right) + \exp\left( \dfrac{\left(r-\frac{t}{350}\right)^{2}}{100}\right)+\dfrac{1}{2}\exp\left( \dfrac{\left(r-\frac{t}{400}\right)^{2}}{150}\right)\\
		\dfrac{3}{2}\exp\left( \dfrac{\left(r-\frac{t}{2000}\right)^{2}}{70}\right)+ \exp\left( \dfrac{\left(r-\frac{t}{350}\right)^{2}}{140}\right)+\dfrac{1}{2}\exp\left( \dfrac{\left(r-\frac{t}{400}\right)^{2}}{210}\right)\\
	\dfrac{3}{2}\exp\left( \dfrac{\left(r-\frac{t}{2000}\right)^{2}}{80}\right)+ \dfrac{4}{5}\exp\left( \dfrac{\left(r-\frac{t}{200}\right)^{2}}{160}\right)+\dfrac{3}{5}\exp\left( \dfrac{\left(r-\frac{t}{300}\right)^{2}}{240}\right)
	\end{array}\right),
\end{equation}

 The nonlocal system (\ref{radialmodeltest}) can be considered a particular case of the general system (\ref{radialmodel}) by setting specific parameter values. Specifically, we choose $\widetilde{\alpha_{p}}(t,r) = \widetilde{\alpha_{s}}(t,r) = r$, $\widetilde{\gamma_{p}}(r) = \widetilde{\gamma_{s}}(r) = 1$, $a_{1} = 1$, $a_{2} = 1$, $\lambda(t) = t$, $M_{s} = \eta = 1$, and $D = \dfrac{r}{100}$. These parameters were selected to allow the radial convolution to be calculated analytically. This is important because radial systems often exhibit singularities at $r = 0$. To address these singularities, the velocity functions $\widetilde{\alpha_{p}}$, $\widetilde{\alpha_{s}}$, and the diffusion parameter $D$ are specially chosen.\\

To numerically solve the system, we discretize the domain $[0, R]$ using a uniform grid. The grid points are defined as $r_{i} = ih$, for $i=1,\dots,N$. Where $h$ is the spatial step size given by $h = \dfrac{R}{N}$, and $N$ being the number of discretization nodes.\\

We then compare the approximate numerical solution to the analytical solution using the relative error metrics defined as follows:

\begin{equation*}
E^{\infty}_{n} = \frac{\left\|\widetilde{U}(t^{n}) - \widetilde{U_{h}}(t^{n})\right\|_{\infty}}{\left\| \widetilde{U}(t^{n})\right\| _{\infty}} = \frac{\sup_{0 \leq i \leq N}\left\|\widetilde{U}\left(t^{n}, r_{i}\right) - \widetilde{U_{h}}(t^{n}, r_{i})\right\|}{\sup_{0 \leq i \leq N}\left\|\widetilde{U}\left(t^{n}, r_{i}\right)\right\|},
\end{equation*}
\begin{equation*}
\text{and} \qquad E^{\infty} = \frac{\left\|\widetilde{U} - \widetilde{U_{h}}\right\|_{\infty}}{\|\widetilde{U}\|_{\infty}} = \frac{\sup_{n}\left\|\widetilde{U}(t^{n}) - \widetilde{U_{h}}(t^{n})\right\|_{\infty}}{\sup_{n}\left\| \widetilde{U}(t^{n})\right\|_{\infty}}.
\end{equation*}
These relative error metrics provide a means to quantitatively assess the accuracy of the numerical solution compared to the analytical solution. The following table presents the execution times and relative errors for different numbers of nodes:

\begin{table}[H]
\centering
\begin{tabular}{||c || c| c |c|c|c|c||}
	\hline \hline
$N$ & $50$ & $100$ &  $150$ & $200$ & $250$  & $300$  \\
\hline
$E^\infty$ &$3.344e-03$& $1.456e-03$ & $6.610e-04$& $3.707e-04$ &
$2.440e-04$ & $1.649e-04$\\
\hline
CPU Time & $18.96s$ & $48.56s$ &$188.53s$ &$422.28s$ &$853.93s$ &$1619.54s$ \\
	\hline 	\hline
\end{tabular}
	\caption{Execution time and relative error for different node numbers for $R=1$, and $T=10$.}
\label{table:e_inf2}
\end{table}

The following figure illustrates the comparison between the approximate solution $\widetilde{U}_{h}$ and the exact solution $\widetilde{U}_{ex}$ of our nonlocal system at various time points: $t = 0, t = 25, t = 50, t = 75$, and $t = 100$.
\begin{figure}[H]
	\centering
	\includegraphics[width=\linewidth]{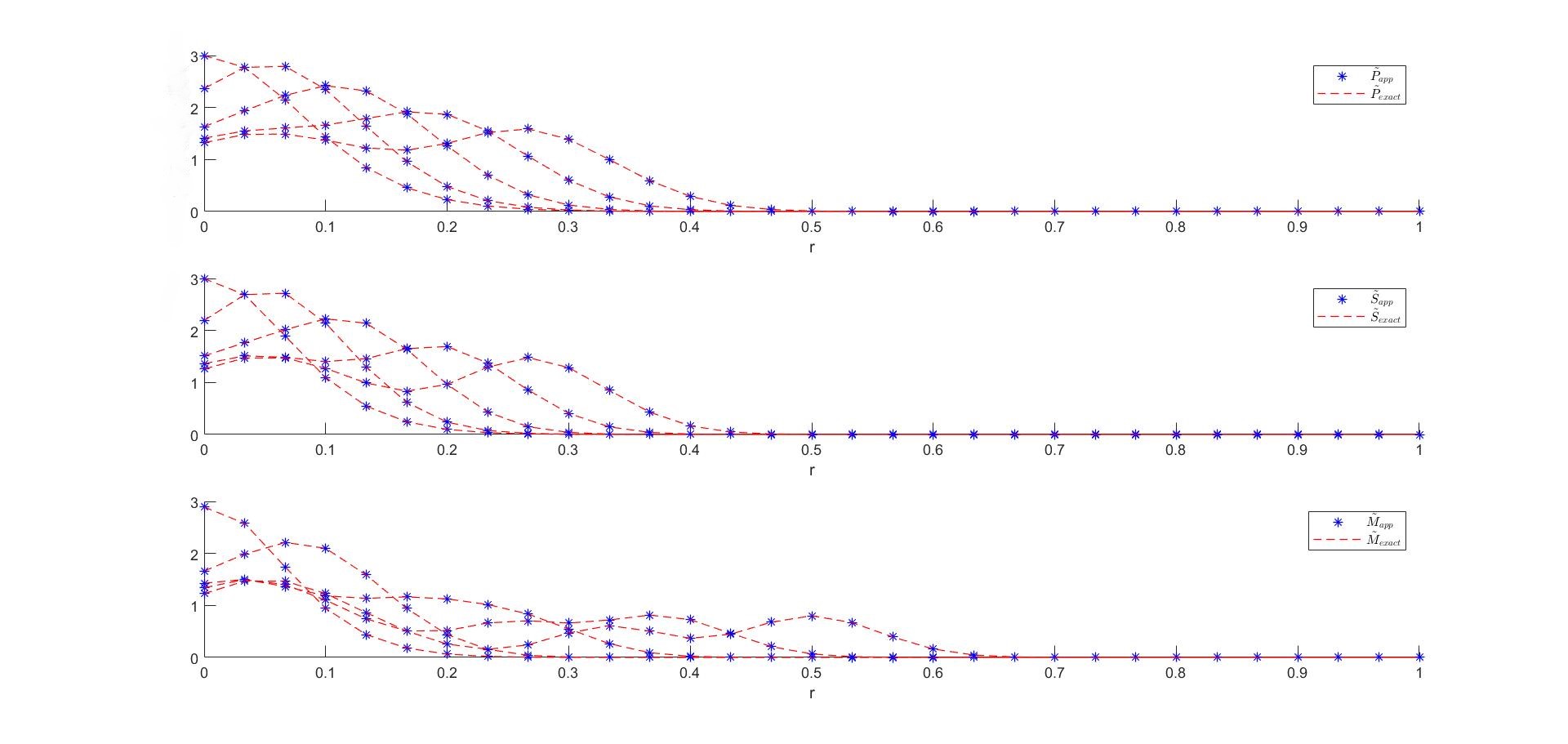}
	\caption{Approximate solution $\widetilde{U}_{h}$ and Exact solution $\widetilde{U}_{ex}$ for $t=0,\,\, t=25,\, \, t=50,\, \, t=75,\, \, t=100$.}
	\label{fig:figg}
\end{figure}
\begin{figure} 
	\centering
	\begin{minipage}{.55\textwidth}
		\centering
		\includegraphics[width=\linewidth]{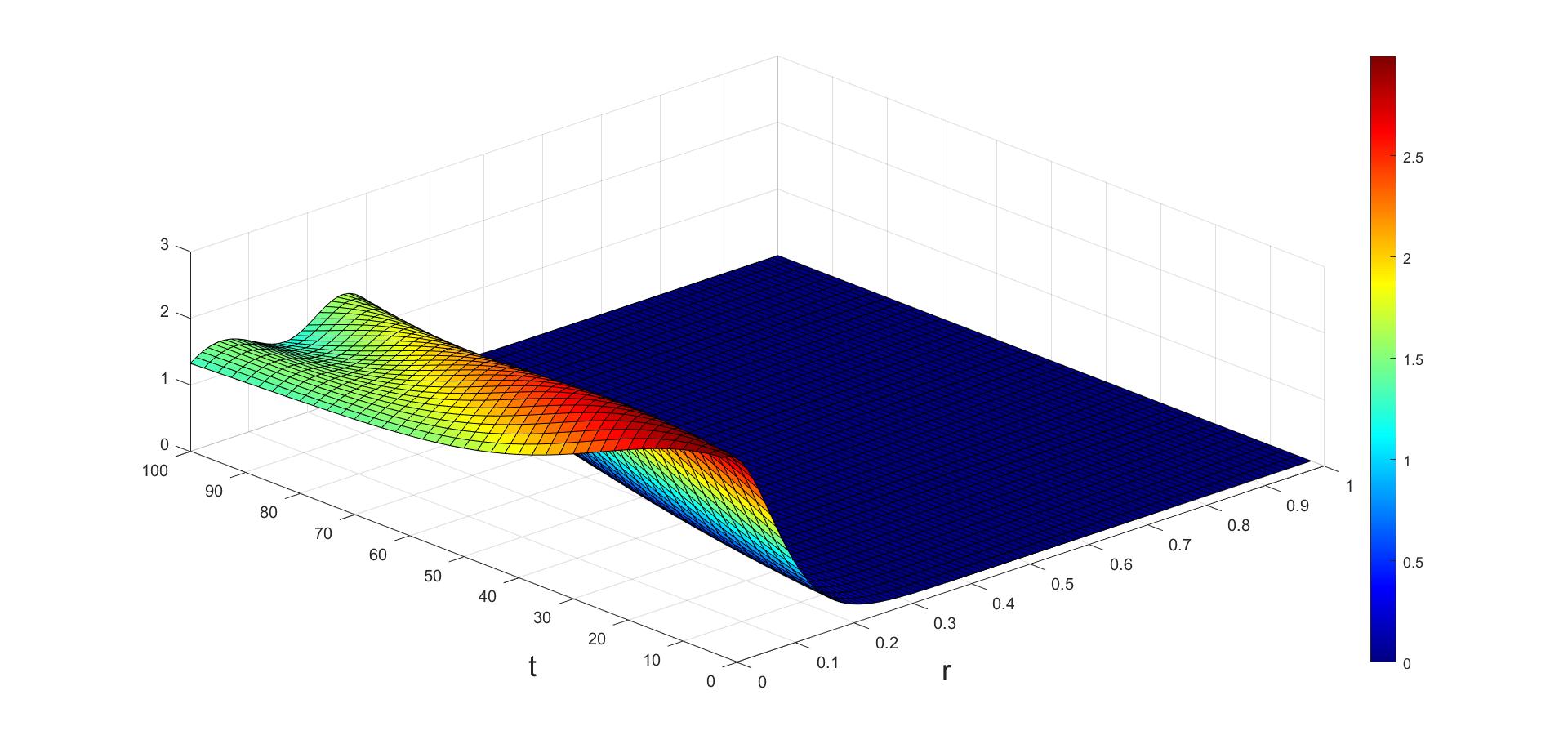}
		\captionof{figure}{Approximation of $\widetilde{P}$ on $[0,1]\times[0,100]$.} 
		\label{fig:test1}
	\end{minipage}%
	\begin{minipage}{.55\textwidth}
		\centering
		\includegraphics[width=\linewidth]{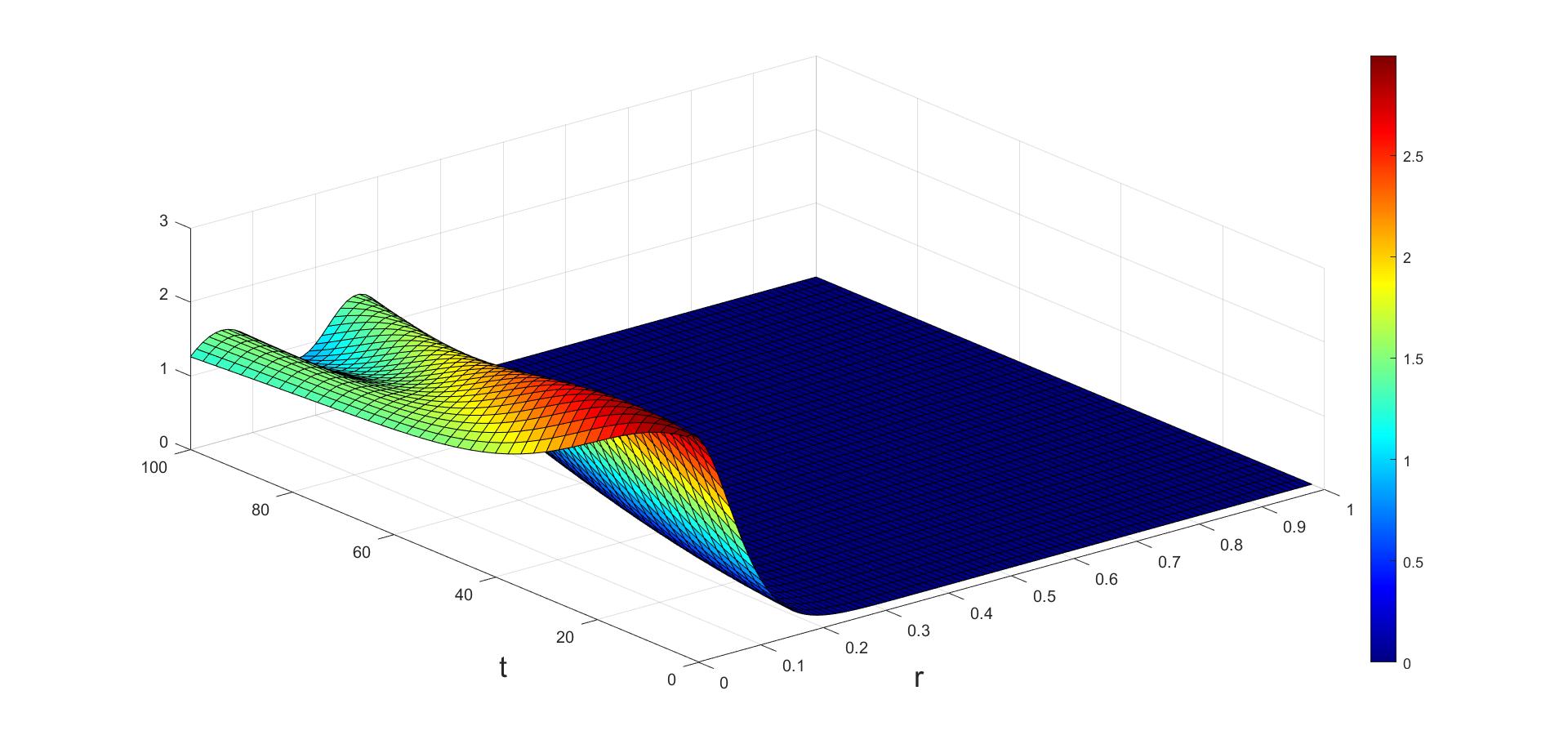}
		\captionof{figure}{Approximation of $\widetilde{S}$ on $[0,1]\times[0,100]$.}
		\label{fig:test2}
	\end{minipage}
\end{figure}
\begin{figure}[H]
	\centering
	\includegraphics[width=.6\linewidth]{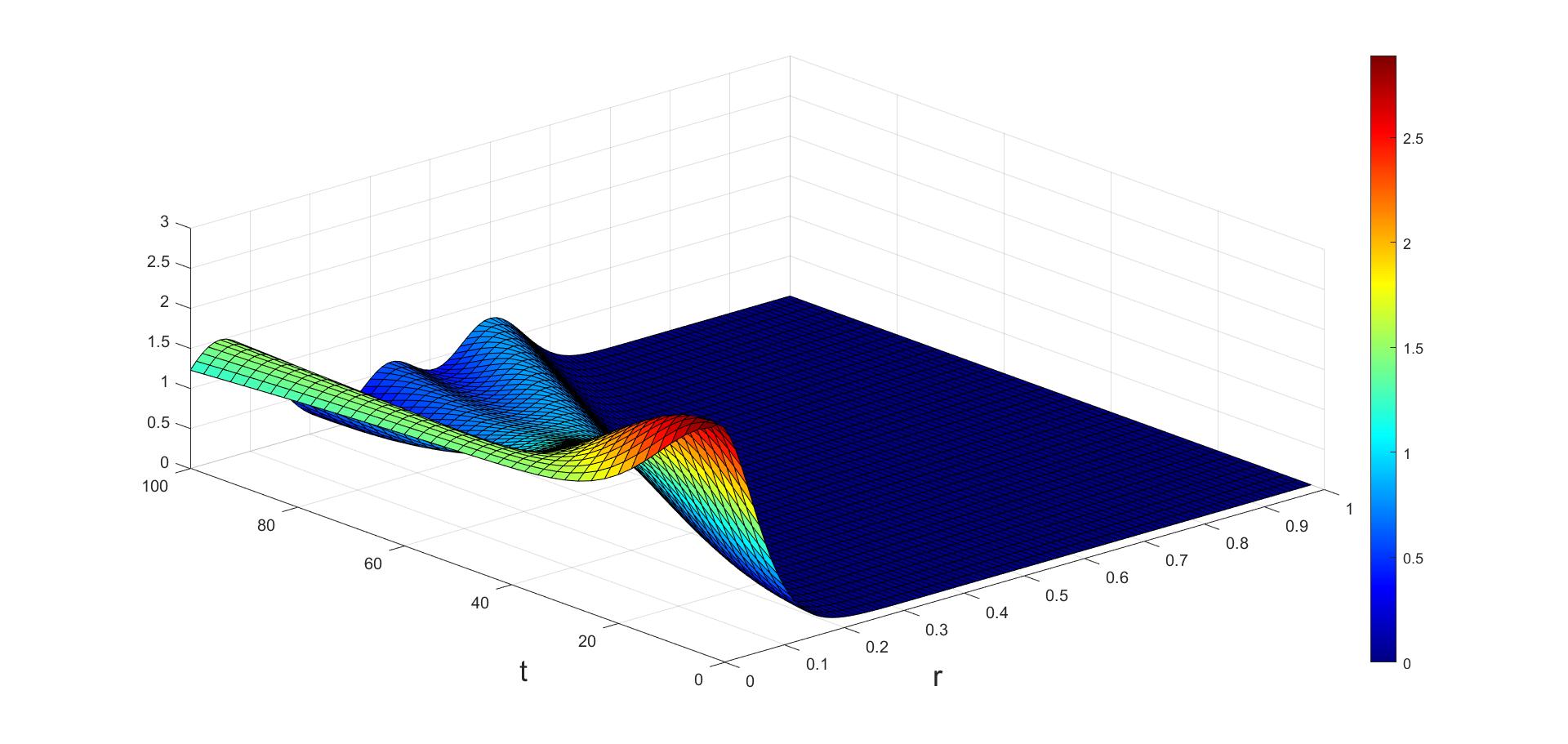}
	\caption{Approximation of $\widetilde{M}$ on $[0,1]\times[0,100]$.}
	\label{fig:prob1_6_2}
\end{figure}
For a range of time values, a comparison between the numerical solution, denoted as $\widetilde{U}_{a}$, and the exact solution, denoted as $\widetilde{U}_{ex}$, is conducted, as illustrated in Figure \ref{fig:figg}. In this figure, it is evident that the numerical solution approximates the exact solution with high accuracy across the different time instances (t=0,\,\, t=25,\,\, t=50,\,\, t=75,\,\, t=100). This demonstrates the effectiveness of the numerical method in closely replicating the exact solution over the specified time intervals.

\begin{table}[H]
\centering
\begin{tabular}{||c||c|c|c|c|c||c}
	\hline \hline
       $N$ &$100$&$200$&$300$&$400$&$500$\\
        \hline
	$h$& $0.01$  &0.005  &$0.0033$&$  0.0025$&$0.002$ \\
 \hline
	$\Delta t$& $0.005$  &$0.0025$   & $0.00165$&$0.00125$& $0.001$ \\
	\hline \hline
t=20	& 6.9801e-04 & 1.7519e-04 &7.7218e-05&  4.3550e-05& 2.7949e-05 \\
	\hline
t=40	& 4.1530e-04 & 1.0432e-04 &5.3297e-05&  2.6102e-05  &1.67076e-05  \\
	\hline
t=60	& 8.0758e-04 & 2.0209e-04 &9.0763e-05&    4.9907e-05& 3.2346e-05 \\
	\hline
t=80	& 1.4171e-03 &  3.5577e-04&1.5894e-04&    8.9162e-05& 5.6989e-05 \\
	\hline
t=100	&1.8271e-03  &  4.5819e-04&2.0402e-04&    1.1462e-04&7.3364e-05\\
	\hline
$E^{\infty}$ & 5.0599e-04 &  1.2579e-04 &7.0742e-05&    3.1597e-05&   2.0176e-05 \\
	\hline \hline
\end{tabular}
	\caption{Relative error $E^{\infty}$.}
\label{table:e_inf}
\end{table}

Table \ref{table:e_inf} presents various values of the relative error $E_{n}^{\infty}$ for $\Delta t=\frac{h}{2}$ with different spatial step sizes $h$. The data suggest that as $h$ decreases, the computational error also decreases. Although proving the theoretical convergence of the Backward Differentiation Formula (BDF) method for a nonlinear Ordinary Differential Equation (ODE) is challenging, the computed results indicate a reduction in error with smaller spatial steps. This observation is consistent with existing literature on numerical analysis and the stability of numerical integration methods. The known stability of the BDF method in stiff problems corroborates the observed error reduction trends with decreasing $h$ values.

\subsection{Numerical Test for Tumor Growth}\label{sub:6.2}
In this subsection, we present numerical simulations of the nonlocal tumor growth model (\ref{model}), under the scenario of no medical intervention. 

The initial conditions for our simulations are defined as follows:

\begin{equation*}
  \widetilde{P}(0,r) = \dfrac{0.3e^2}{e^2 + e^{20r}}, \quad \widetilde{S}(0,r) = \dfrac{0.8e^2 + e^{20r}}{e^2 + e^{20r}}, \quad \widetilde{M}(0,r) = 1.
\end{equation*}
At the boundary radius $R$, the conditions are specified as follows:

\begin{equation*}
\widetilde{P}(t,R) = 0, \quad \widetilde{S}(t,R) = 1, \quad \widetilde{M}(t,R) = 1.
\end{equation*}
The velocities and kernel functions incorporated in our model are expressed as follows.

\begin{equation*}
\widetilde{\alpha_{p}}(t,r) = \dfrac{2e}{e + e^{10r}}, \quad \widetilde{\alpha_{s}}(t,r) = \dfrac{0.2e^{7.5}}{e^{7.5} + e^{15r}}, \quad \widetilde{\gamma_{p}}(r) = 0.2\exp\left( -\frac{r}{50}\right), \quad \widetilde{\gamma_{s}}(r) = \dfrac{1}{100}\exp\left( -\frac{r}{150}\right).
\end{equation*}
We used a hypoxia threshold value of $M_{th} = 0.5$. The growth factor parameter, $\kappa$, is set at 0.02. The model also integrates a diffusion coefficient $D$ of 0.8. Furthermore, the nutrient consumption rate, $\eta$, is 1.5, and the blood vessel creation rate is characterized by $M_{s} = 1$.

The results of our numerical simulations are demonstrated in Figure \ref{fig:tumor_growth}.

\begin{figure}[H]
	\centering

	\begin{subfigure}[b]{0.32\textwidth}
		\centering
		\includegraphics[width=\textwidth]{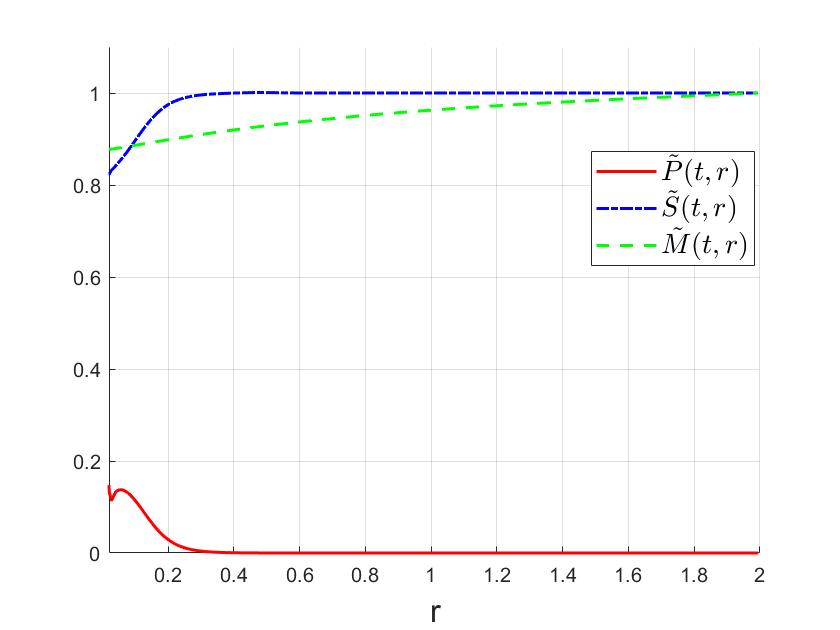}
		\caption{$t=5$}
		\label{fig:t5}
	\end{subfigure}
	\hfill
\begin{subfigure}[b]{0.32\textwidth}
		\centering
		\includegraphics[width=\textwidth]{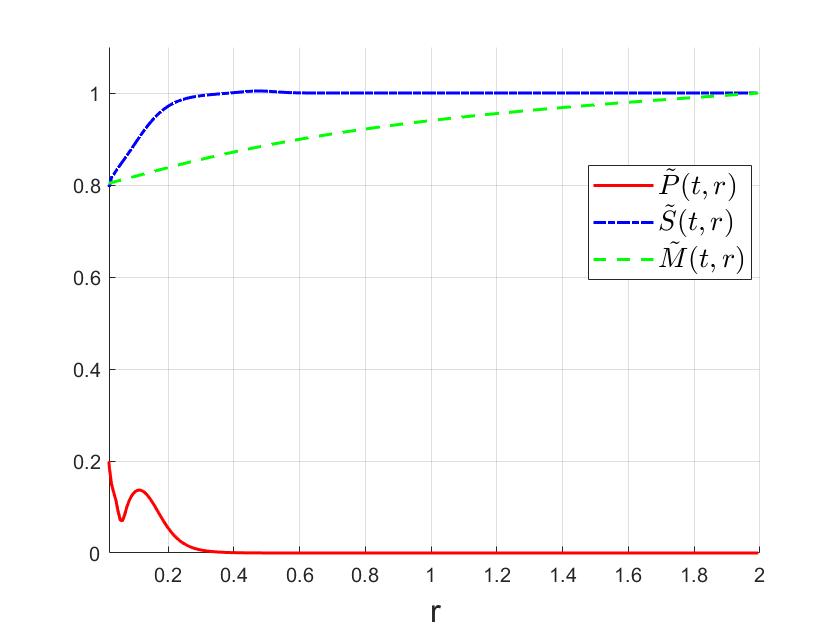}
		\caption{$t=20$}
		\label{fig:t20}
	\end{subfigure}
	\hfill
 \begin{subfigure}[b]{0.32\textwidth}
		\centering
		\includegraphics[width=\textwidth]{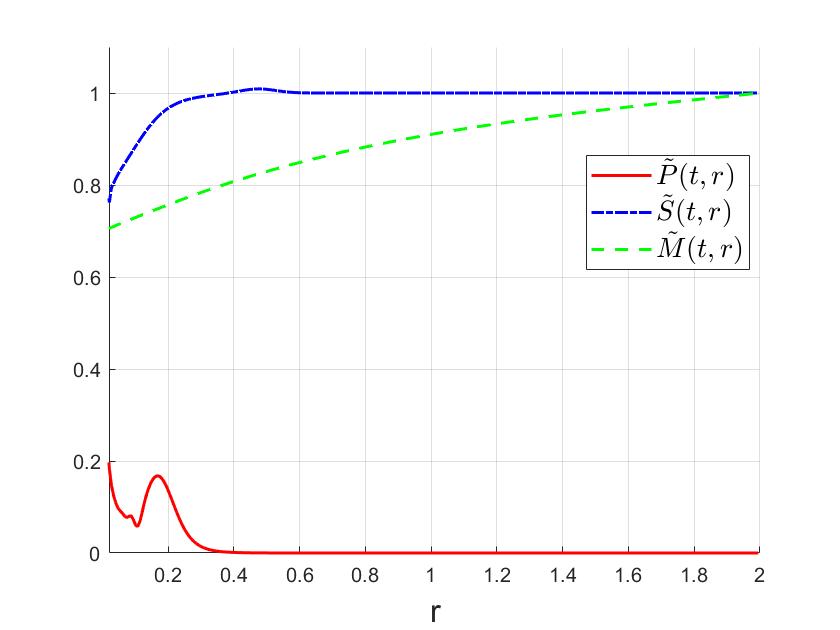}
		\caption{$t=40$}
		\label{fig:t40}
	\end{subfigure}
	\hfill
\begin{subfigure}[b]{0.32\textwidth}
		\centering
		\includegraphics[width=\textwidth]{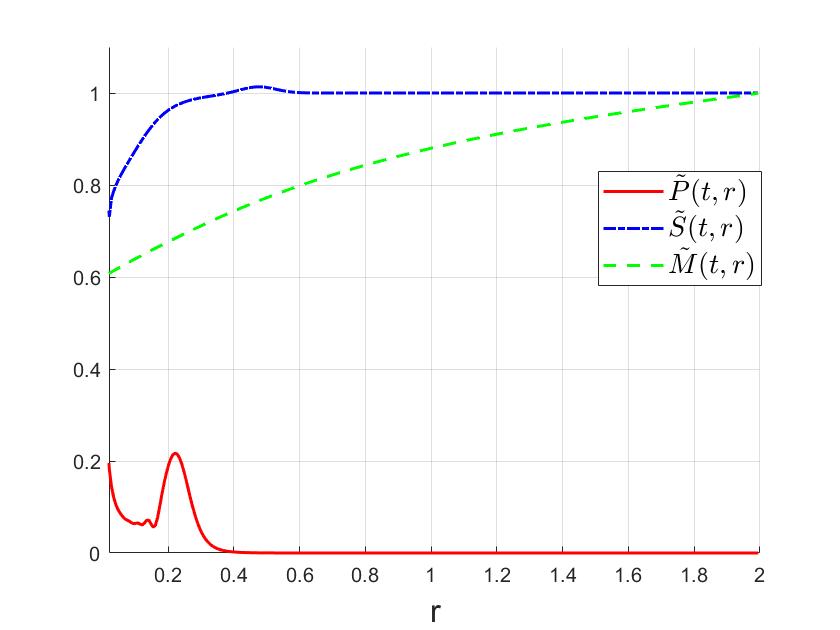}
		\caption{$t=60$}
		\label{fig:t60}
	\end{subfigure}
	\hfill
\begin{subfigure}[b]{0.32\textwidth}
		\centering
		\includegraphics[width=\textwidth]{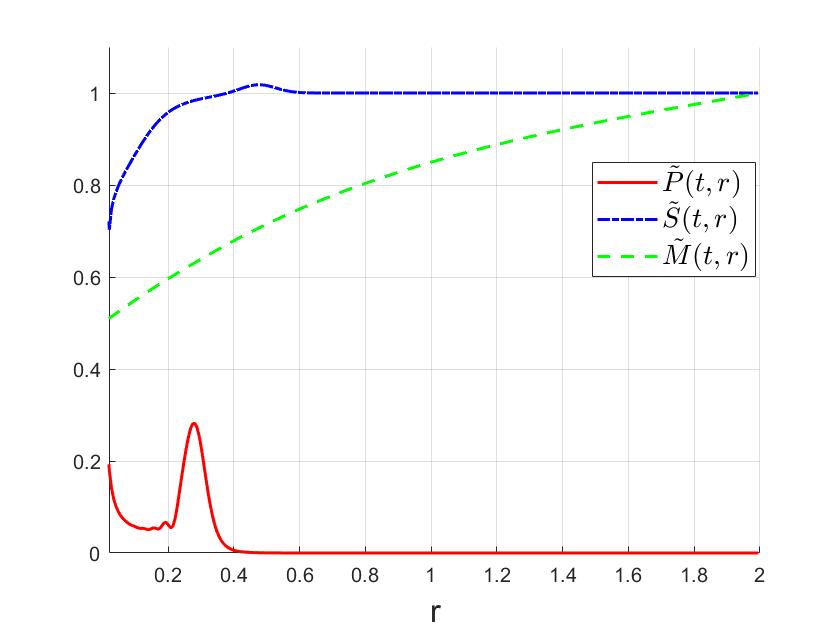}
		\caption{$t=80$}
		\label{fig:t80}
	\end{subfigure}
	\hfill
\begin{subfigure}[b]{0.32\textwidth}
		\centering
		\includegraphics[width=\textwidth]{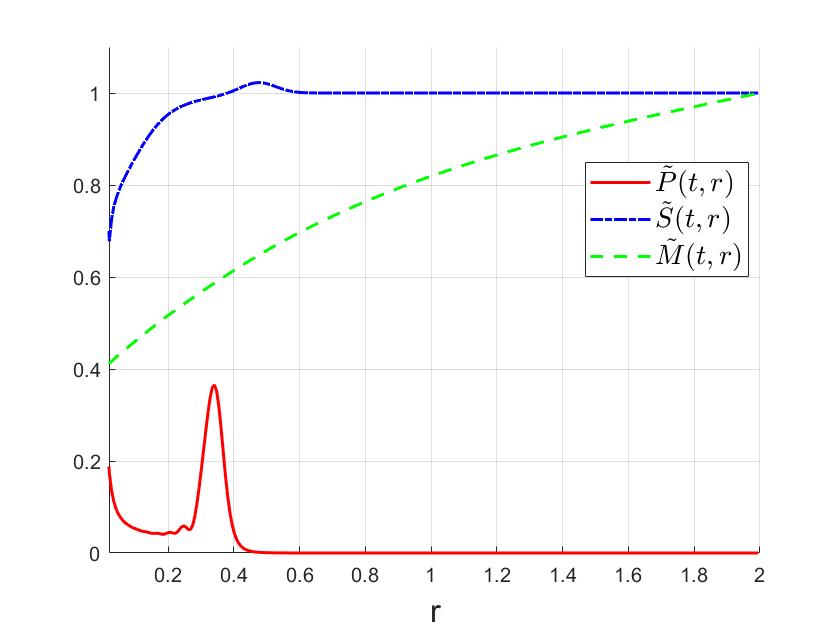}
		\caption{$t=100$}
		\label{fig:t100}
	\end{subfigure}
\caption{Radial tumor cells density $\widetilde{P}$, healthy cells density $\widetilde{S}$, and the concentration of nutriment and oxygen $\widetilde{M}$ for different time instances.}
 
	\label{fig:tumor_growth}
\end{figure}

\begin{figure}[H]
	\centering
	\begin{subfigure}[b]{0.32\textwidth}
		\centering
		\includegraphics[width=\textwidth]{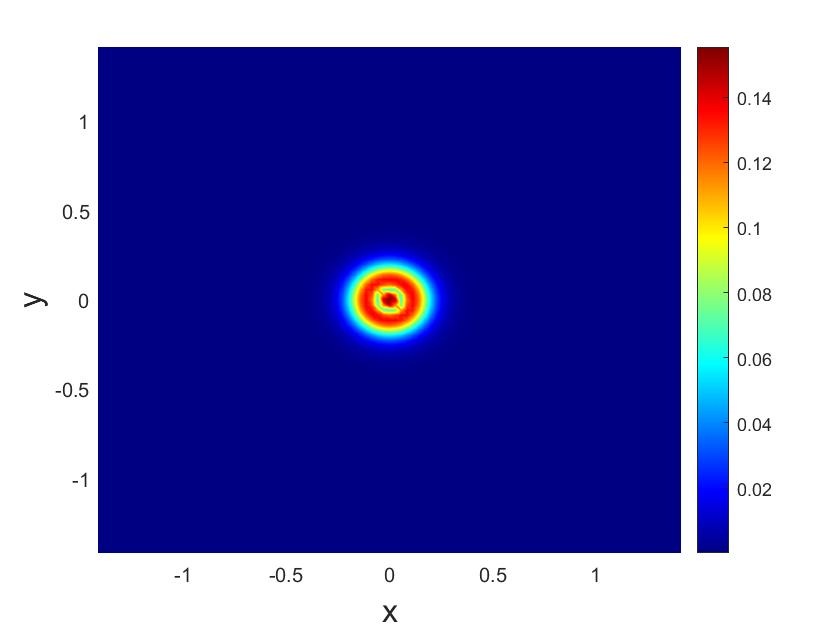}
		\caption{$t=0$}
		\label{fig:y equals x}
	\end{subfigure}
	\hfill
	\begin{subfigure}[b]{0.32\textwidth}
		\centering
		\includegraphics[width=\textwidth]{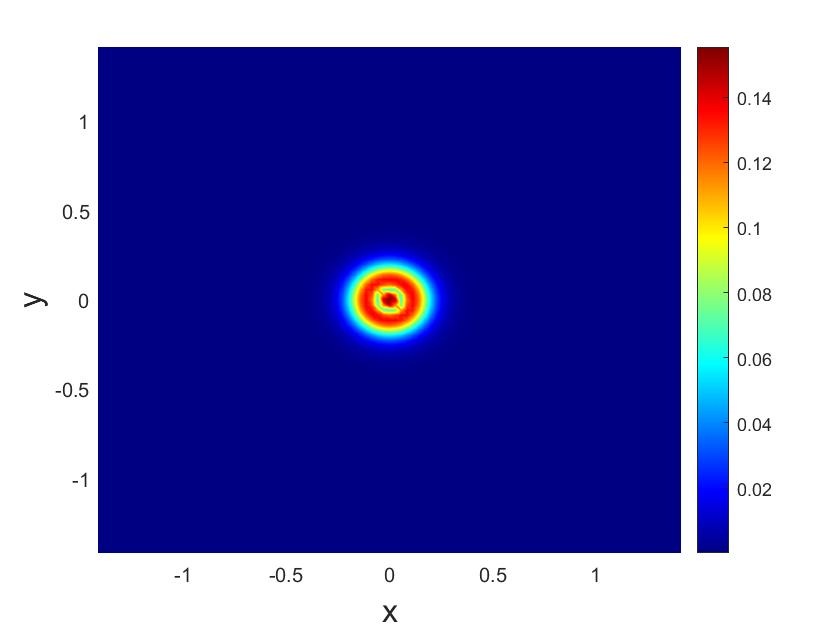}
		\caption{$t=20$}
		\label{fig:three sin x}
	\end{subfigure}
	\hfill
	\begin{subfigure}[b]{0.32\textwidth}
		\centering
		\includegraphics[width=\textwidth]{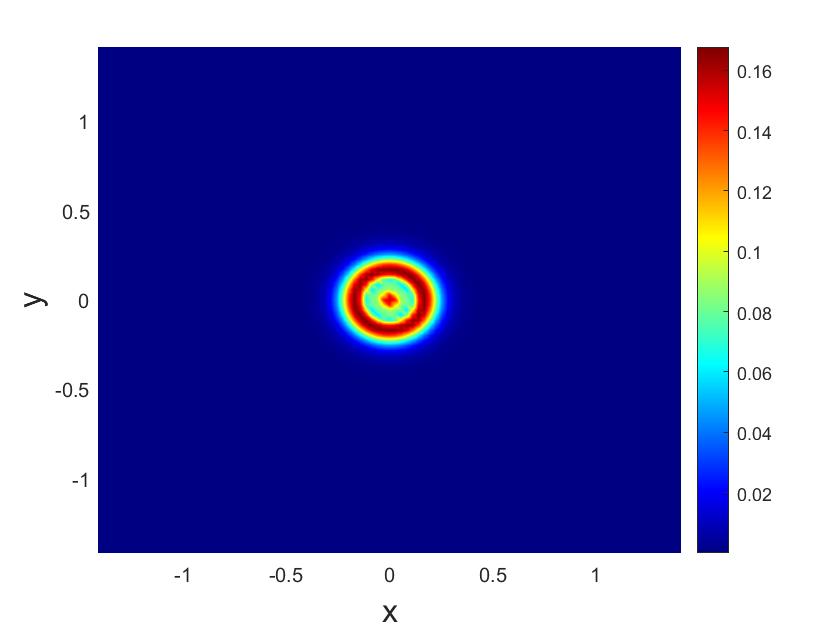}
		\caption{$t=40$}
		\label{fig:five over x}
	\end{subfigure}
		\begin{subfigure}[b]{0.32\textwidth}
		\centering
		\includegraphics[width=\textwidth]{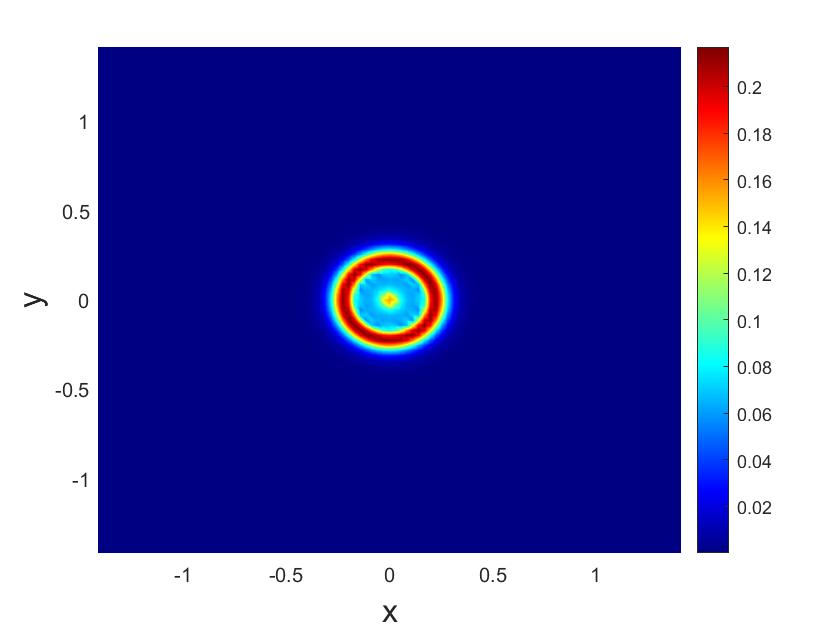}
		\caption{$t=60$}
		\label{fig:y equals x}
	\end{subfigure}
	\hfill
	\begin{subfigure}[b]{0.32\textwidth}
		\centering
		\includegraphics[width=\textwidth]{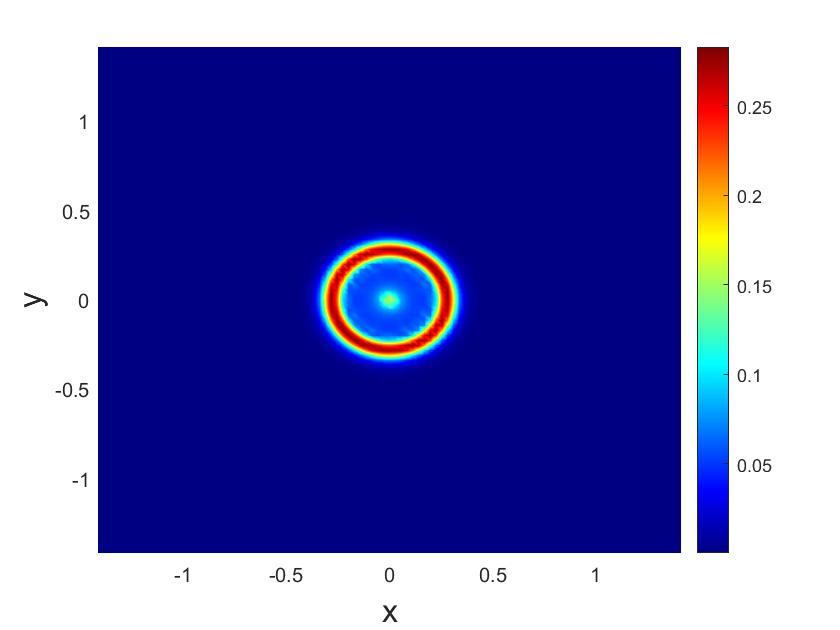}
		\caption{$t=80$}
		\label{fig:three sin x}
	\end{subfigure}
	\hfill
	\begin{subfigure}[b]{0.32\textwidth}
		\centering
		\includegraphics[width=\textwidth]{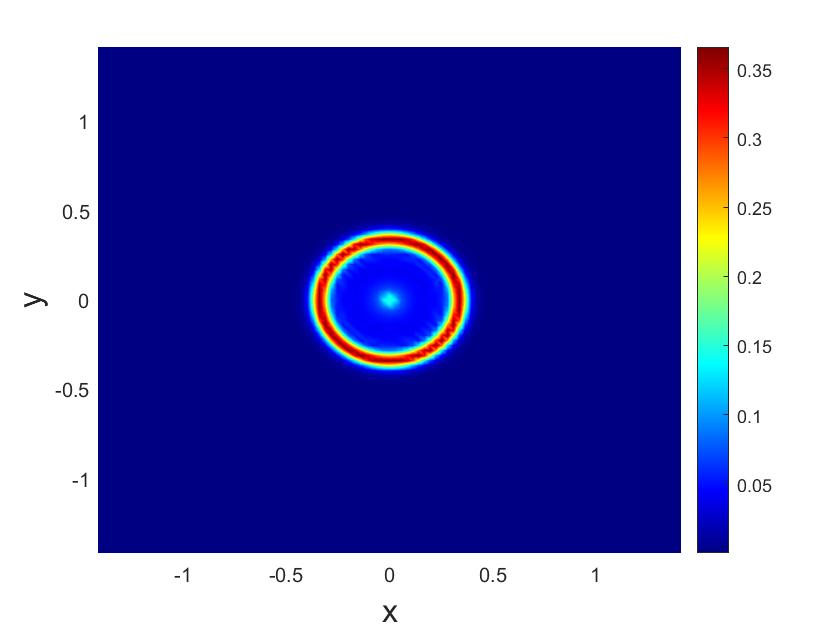}
		\caption{$t=100$}
		\label{fig:five over x}
	\end{subfigure}
	
	\caption{Tumor cells density ${P}$ for different time instances.}
	\label{fig:P}
\end{figure}

\begin{figure}[h]
	\centering
\includegraphics[width=0.5\textwidth]{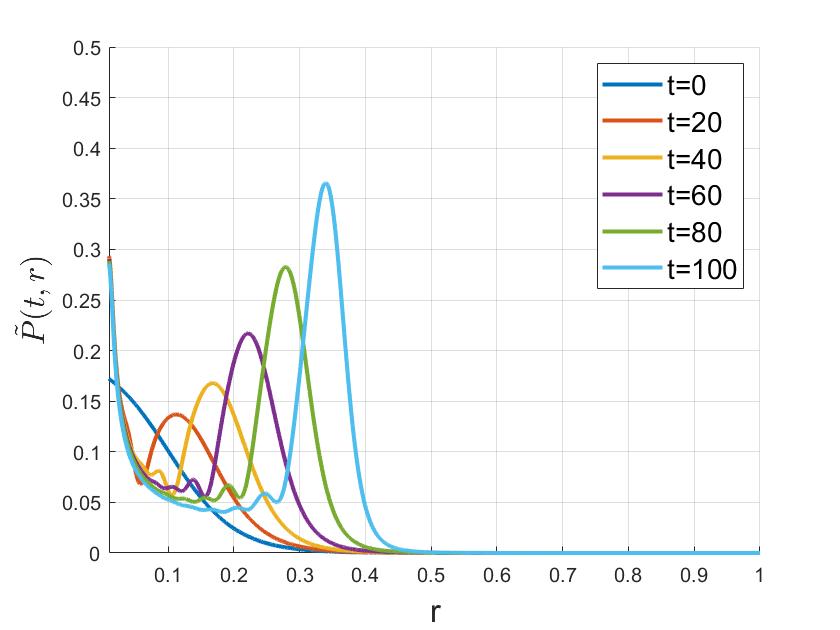}
	\caption{Radial tumor cells density for different time instances.}
\label{fig:tumor1}
\end{figure}
Figures \ref{fig:P} and \ref{fig:tumor1} illustrates the growth of the tumor core, particularly at its edges where higher nutrient concentrations are present, facilitating cellular proliferation. The presence of abundant nutrients at the periphery enables the cells to rapidly divide and expand outward. In contrast, the center of the tumor exhibits a decrease in tumor cell density due to insufficient oxygen supply, leading to necrosis. The lack of oxygen, known as hypoxia, results in cell death and the formation of a necrotic core. This process highlights the dynamic nature of tumor growth, where peripheral areas continue to expand while central regions undergo cell death. Furthermore, the hypoxic conditions in the tumor center can lead to the release of signaling molecules that may promote angiogenesis or new blood vessel formation in an attempt to supply the tumor with more oxygen and nutrients.

\section*{Conclusion} \label{conclusion}
The process of tumor growth remains a complex subject to study, predominantly due to the multifaceted nature of the factors that influence it. A detailed understanding of the mechanisms governing tumor growth is essential for the advancement of effective anticancer strategies.

In this paper, we have presented an innovative continuum model of tumor growth, which includes both cell-cell and cell-matrix adhesion and interactions. This model is characterized by the integration of nonlocal terms within a system of partial differential equations (PDEs). We have successfully demonstrated the existence and uniqueness of solutions through the application of semigroup properties and Banach's fixed-point theorem, in combination with results established by Keimer in \cite{keimer_existence_2018}. Furthermore, we have developed a numerical scheme for resolving the radial nonlocal system. This scheme employs B-splines for spatial discretization and Backward Differentiation Formula (BDF) for temporal discretization. Extensive numerical tests have been conducted to validate the model's accuracy, and we have demonstrated experimental results related to tumor growth.

The use of nonlocal balance equations in our model offers several significant advantages over traditional local approaches. For example, nonlocal models can avoid the formation of shock waves and singularities that often challenge local hyperbolic systems. This regularizing effect leads to smoother solutions and may prevent the breakdown of numerical schemes, which is a common issue in local models of complex phenomena. The efficient numerical solution of nonlocal systems has potential applications in many fields such as population dynamics, material science, and climate modeling, where long-range effects play a crucial role and where avoiding discontinuities is critical. 

In the future, our research aims to optimize the factors of medical drugs within the system (\ref{model}), using the theory of optimal control for nonlocal balance equations. This will help design new strategies that minimize the auxiliary effects of treatment and maximize its efficacy. Using mathematical models, we intend to precisely adjust drug dosages, timings, and combinations to achieve the best therapeutic outcomes while reducing negative side effects.

\section*{Acknowledgment}
This work was supported by the Centre National pour la Recherche Scientifique et Technique (CNRST) of Morocco and the Laboratoire de Mathématiques Pures et Appliquées (LMPA) at the Université du Littoral Côte d'Opale (ULCO), France.
\newpage
\bibliography{references}
\bibliographystyle{unsrt} 

\include{annex}
			
		\end{document}